\documentclass[11pt,oneside,a4paper,reqno]{amsart}

\usepackage[utf8]{inputenc}
\usepackage[T1]{fontenc}
\usepackage[english]{babel}

\usepackage{newtxtext}

\usepackage[
    activate={true,nocompatibility},
    final,
    tracking=true,
    spacing=true,
    stretch=15,
    shrink=15
]{microtype}
\microtypecontext{spacing=nonfrench}
\SetTracking{encoding=*,family=*}{-20}

\usepackage[
    a4paper,
    margin=1.1in,
    top=1.2in,
    bottom=1.2in,
    headheight=15pt,
    headsep=0.4in,
    footskip=1.2cm
]{geometry}

\usepackage[dvipsnames,table,svgnames]{xcolor}

\definecolor{academicblue}{HTML}{1A365D}
\definecolor{winfull}{HTML}{7B2D3A}
\definecolor{wintrade}{HTML}{C99AA2}
\definecolor{hlred}{HTML}{9B2C2C}
\definecolor{rowgray}{gray}{0.95}

\usepackage{amsmath,amssymb,amsthm}
\usepackage{bm}

\usepackage{booktabs}
\usepackage{enumitem}
\setlist[enumerate]{label=(\roman*),leftmargin=*,align=left}
\usepackage{arydshln}

\usepackage{algorithm}
\usepackage{algpseudocode}
\algrenewcommand\algorithmiccomment[1]{\hfill\textcolor{gray!70}{$\triangleright$ #1}}

\usepackage{caption}
\usepackage{todonotes}
\usepackage{comment}
\usepackage[normalem]{ulem}

\usepackage{hyperref}
\hypersetup{
    colorlinks=true,
    citecolor=academicblue,
    linkcolor=academicblue,
    urlcolor=academicblue,
    breaklinks=true
}

\newcommand{\vf}[1]{\textcolor{winfull}{#1}}
\newcommand{\vt}[1]{\textcolor{wintrade}{#1}}

\newcommand{\jball}{B_{\lambda,1}}
\newcommand{\ji}{i}
\newcommand{\jqtwo}{\mathfrak{q}}
\newcommand{\jq}{\mathfrak{q}}
\newcommand{\jetwo}{\mathfrak{e}}
\newcommand{\je}{\mathfrak{e}}
\newcommand{\jf}{\mathfrak{f}}

\newcommand{\cpt}{\text{\upshape)}}
\newcommand{\opt}{\text{\upshape(}}

\newcommand{\RR}{\mathbb{R}}
\newcommand{\NN}{\mathbb{N}}

\newcommand{\CC}{\mathbb{C}}

\newtheorem{definition}{Definition}
\newtheorem{theorem}{Theorem}
\newtheorem{lemma}{Lemma}
\newtheorem{proposition}{Proposition}

\theoremstyle{remark}
\newtheorem{remark}{Remark}

\begin{document}

\title[A Fixed-Point Construction of the
Elementary Transcendental Functions]{A Fixed-Point Construction of the\\
Elementary Transcendental Functions}

\author{Fran\c{c}ois Alouges}
\address{{\textsc{Fran\c{c}ois Alouges.}} $^1$Universit\'e Paris Saclay, Universit\'e Paris Cit\'e, ENS Paris-Saclay, CNRS, SSA, INSERM, Centre Borelli, 4 avenue des Sciences, Gif-sur-Yvette 91190, France.
$^2$Institut Universitaire de France.}
\email{falouges@ens-paris-saclay.fr}

\author{Giovanni Di Fratta}
\address{{\textsc{Giovanni Di Fratta.}} $^1$Dipartimento di Matematica e Applicazioni ``R. Caccioppoli'', Universit\`a degli Studi di Napoli ``Federico II'', Via Cintia, Complesso Monte S. Angelo, 80126 Naples, Italy}
\email{giovanni.difratta@unina.it}

\author{Alberto Fiorenza}
\address{{\textsc{Alberto Fiorenza.}} $^1$Dipartimento di Architettura, Universit\`a di Napoli, Via Monteoliveto 3, 80134 Napoli, Italy.
$^2$Istituto per le Applicazioni del Calcolo ``Mauro Picone'', Sezione di Napoli, Consiglio Nazionale delle Ricerche, Via Pietro Castellino 111, 80131 Napoli, Italy}
\email{fiorenza@unina.it}

\author{Renato Fiorenza}
\address{{\textsc{Renato Fiorenza.}} Accademia di Scienze Fisiche e Matematiche, Via Mezzocannone 8, 80134 Napoli, Italy}
\email{renato.fiorenza.senior@gmail.com}

\begin{abstract}
We present a unified fixed-point construction of the elementary transcendental
functions, encompassing the real exponential, the complex exponential (sine and
cosine), and the natural logarithm. Each function is characterized as the
unique solution of a duplication identity established through the Banach
contraction principle. These foundational identities are $\je(2x)=\je^2(x)$ for
the exponentials, and $\log(x^{2})=2\log x$ for the logarithm. Because a direct
iteration of these identities is numerically unstable, owing to local
expansiveness at the target, the central idea transfers the analysis to a
residual function, on which the operator becomes a strict contraction with an
explicit convergence rate. Beyond its theoretical economy, which dispenses with
differential equations and power series, this characterization translates into
efficient algorithms for the machine evaluation of elementary functions: the
underlying framework yields floating-point kernels whose accuracy and iteration depth are governed by the theoretical contraction rate. We also present a computational study showing that, in a throughput-bound vectorized regime, these
kernels are competitive with standard production libraries, and in favorable
configurations exceed them, with the sine--cosine kernel faster at every tested iteration depth. These implementations operate without lookup tables or memory
traffic, an architectural advantage for modern high-performance and
energy-efficient computing.
\end{abstract}

\maketitle
\section{Introduction}\label{sec:intro}
The elementary transcendental functions, namely the exponential, the logarithm, 
and the trigonometric functions, are among the most frequently evaluated primitives 
in scientific computing. Large-scale numerical simulations routinely require billions 
of evaluations and the routines that serve those calls, refined over decades in 
libraries such as the system \texttt{libm}, Intel's SVML, and the vectorised library 
SLEEF~\cite{shibata2020sleef}, almost invariably rest on the same architectural 
choice: a precomputed lookup table brings the argument into a narrow interval, 
after which a short polynomial finishes the evaluation. 
Lookup tables are efficient in scalar code but poorly suited to modern Single Instruction Multiple Data (SIMD) architectures. SIMD execution achieves high throughput by evaluating multiple independent arguments simultaneously, but different vector lanes generally require different table entries. The resulting irregular memory accesses are implemented through gather operations, which incur substantially higher latency than contiguous loads and increase memory traffic. As a result, arithmetic pipelines are often stalled while waiting for data, reducing the throughput that SIMD parallelism is intended to deliver.

In this work, we develop a table-free construction of the elementary transcendental functions based on a structural characterisation rather than on polynomial approximation. Our starting point is the observation that each function is uniquely determined by a duplication identity, namely a functional equation relating the value at an argument to the value at either half or twice that argument, together with a natural normalisation condition. 
 
For the real and complex exponentials, this characterisation is captured by the following axioms:
\begin{enumerate}[label=$(\mathbf{A_{\arabic*}})$, leftmargin=*, align=left]
    \item \label{ax:A1} \emph{Duplication identity}:
    \begin{equation}
    \je(2x) = \je(x)^2 \quad \text{for all } x \in \RR, \label{eqax:A1}
    \end{equation}
    \item \label{ax:A2} \emph{Initial conditions}: 
    \begin{equation}
    \je(0) = 1,\qquad \je'(0)=\gamma\text{ with } \gamma \in \{1, \ji\}. \label{eqax:A2}
    \end{equation}
\end{enumerate}
The parameter $\gamma$ explicitly selects either the real exponential $e^x$ or the imaginary exponential $e^{\ji x}$, and through the latter, the sine and cosine functions. 
 
For the natural logarithm, the analogous structural definition imposes:
\begin{enumerate}[label=$(\mathbf{B_{\arabic*}})$, leftmargin=*, align=left]
    \item \label{ax:B1} \emph{Functional equation}:
    \begin{equation}
    \jf(x^2) = 2\jf(x) \quad \text{for } x > 0, \label{eqax:B1}
    \end{equation}
    \item \label{ax:B2} \emph{Cauchy initial conditions}: 
    \begin{equation}
    \jf(1) = 0,\qquad \jf'(1)=1. \label{eqax:B2}
    \end{equation}
\end{enumerate}
 
The common principle unifying these definitions is that the target function is characterised as the unique fixed point of an operator induced directly by the respective duplication identity. Once this operator is shown to be contractive, the Banach--Caccioppoli theorem immediately yields existence, uniqueness, and a convergent iterative construction on arbitrary compact intervals.

The route to a \emph{contraction} is the one subtle point, and it is the same
for all three functions. Applied naively, the duplication map is unstable: the
linearisation of the squaring $z\mapsto z^{2}$ at the initial value $\je(0)=1$
has derivative $2$, so the operator is locally expansive and the fixed point
repelling. We overcome this by transferring the analysis to a \emph{residual function}.
Writing $\je(x)=1+\gamma x+x^{2}\jq(x)$ isolates the second-order remainder
$\jq$, and the duplication identity, rewritten for $\jq$, becomes a genuine
contraction once the ambient space is equipped with a suitable weighted norm.
The logarithm admits an analogous reformulation based on repeated square-root
reduction rather than repeated squaring. Although the two constructions are
mathematically dual, Section~\ref{sec:computational} shows that this
distinction has important consequences in floating-point arithmetic.
 
The main contribution of this work is that the above characterization is not
merely theoretical, but intrinsically \emph{constructive}. The contraction
argument underlying the existence theorem directly yields an algorithmic
procedure: its two parameters, namely the degree of the polynomial seed and the number of iteration steps, are exchangeable at the explicit rate determined by the contraction.
 
The argument reduction relies exclusively on arithmetic operations, including
exponent extraction and square roots, rather than lookup tables. Consequently,
the resulting computational kernels operate without memory accesses beyond
register-resident data. Section~\ref{sec:computational} substantiates these
claims through a detailed floating-point analysis. On throughput-oriented
vector hardware, where arithmetic latency is amortized across SIMD lanes and the
elimination of table accesses becomes advantageous, the fixed-point kernels are
competitive with a production vector library, and in favorable configurations
outperform it, most consistently for the sine and cosine.
 
Moreover, the iterative formulation naturally exposes an accuracy-performance
trade-off: the target precision can be deliberately relaxed to reduce the
computational cost and achieve higher throughput when full machine accuracy is
not required. By systematically replacing memory traffic with arithmetic
operations, the proposed kernels also offer the potential for improved energy
efficiency, a direction left for future experimental investigation.
 
\smallskip
 
\noindent\textbf{Organization of the paper.} The remainder of the paper is organized as follows. Section~\ref{sec:related} contextualizes our approach within the classical methodologies for defining elementary functions, highlighting the conceptual shift from binary functional equations to unary duplication identities. Section~\ref{sec:fixpointOp} analyzes the fixed point operator governing the real and complex exponentials, establishing the main existence and uniqueness results derived from the Banach--Caccioppoli contraction principle. Section~\ref{sec:proofs} supplies the rigorous proofs for the foundational lemmas of this operator by introducing the concept of admissible weight sequences to strictly control its nonlinear behavior. Section~\ref{sec:banach_proof} carries out the parallel construction for the logarithm, defining its associated linear operator and proving its global strict contraction within an appropriately constructed complete metric space. Finally, Section~\ref{sec:computational} details the practical realization of these theoretical operators, translating them into table-free floating-point kernels and benchmarking their performance in a vectorized regime against the production library SLEEF.

\section{Ways of defining the elementary functions}\label{sec:related}
The definition of an elementary function through a functional relation, rather
than an explicit formula, is a recurring theme in analysis. Before introducing
our own characterization, we briefly review the classical alternatives, both to
place our approach in context and because each relies on a prerequisite that our
method avoids. From a computational perspective, the numerical realization of
the approaches reviewed here has been extensively studied and is now well
understood; see, for example, Muller et al.~\cite{muller2018handbook}.
Accordingly, we take this opportunity to adopt a more pedagogical viewpoint,
focusing on the mathematical structure underlying the different
characterizations before presenting the computational framework developed later
in this work.

\begin{description}[leftmargin=0pt, labelsep=1em, itemsep=1.5ex]
  \item[Power series] One may define $e^{z}$, $\sin x$, and $\cos x$ directly
  as the sums of their power series~\cite{rudin1976,lang1997,ahlfors1979}.
  Although algebraically clean, this route demands a comprehensive understanding of 
  infinite series merely to grasp the definitions. By contrast, characterizing these 
  functions through natural duplication identities offers a distinct pedagogical advantage. 
  The defining statements are elegant and entirely comprehensible without any prior 
  knowledge of sequences or series. From a didactical perspective, one can introduce the 
  functions through the natural properties they satisfy and assert their unique existence 
  early on, deferring the proof of the underlying fixed point theorem to a later stage. 
  Consequently, the foundational definitions remain highly accessible and require only 
  minimal prerequisites in analysis.

  \item[Integral functions and ODEs] Transcendental functions may be defined either as integral functions~\cite{spivak2008,hardy1908} or as the unique solutions to initial value problems. For instance, one commonly defines
  \[
    \log x = \int_{1}^{x}\frac{\mathrm{d}t}{t},\quad x\in(0,+\infty),
    \qquad
    \arcsin x = \int_{0}^{x}\frac{\mathrm{d}t}{\sqrt{1-t^{2}}},\quad x\in[-1,1],
  \]
  with the exponential and trigonometric functions subsequently recovered through functional inversion. Alternatively, the sine and cosine functions are frequently introduced as the unique solutions to the second order linear ordinary differential equation $y'' + y = 0$, subject to the respective initial conditions $y(0)=0, y'(0)=1$ and $y(0)=1, y'(0)=0$. Similar pedagogical remarks apply to these methods. 
  
  %While entirely rigorous, these approaches presuppose a fully developed theory of integration and the inverse function theorem, or the advanced machinery of existence and uniqueness theorems for differential equations, merely to formulate the initial definitions. Consequently, they share the same didactical drawback of requiring a heavy analytic framework before these elementary functions can even be introduced.

  \item[Axiomatic approaches] To permit an early development of calculus,
  several authors introduce the elementary functions by postulating their
  defining properties and deriving the calculus from these, deferring the
  constructive \emph{existence} proof to a later chapter. Apostol, in his
  classic text~\cite[\S2.5]{apostol1967}, characterises the sine and cosine not
  by differential equations but by their fundamental properties. He postulates
  functions $\cos$ and $\sin$ on $\RR$ with $\cos 0 = 1$, $\cos\pi=-1$,
  and $\sin(\pi/2)=1$, satisfying the difference formula
  \[
    \cos(y-x) = \cos y\,\cos x + \sin y\,\sin x
    \qquad(x,y\in\RR),
  \]
  together with the local inequality
  $0<\cos x<\tfrac{\sin x}{x}<\tfrac{1}{\cos x}$ for $0<x<\pi/2$. From these, he derives 
  their derivatives immediately, while postponing the proof that such functions
  \emph{exist} to the chapter on infinite series. Shilov proceeds in the same
  spirit~\cite[Ch.~5]{shilov1996}, defining $\sin$ and $\cos$ as functions
  satisfying the Pythagorean identity, the addition formulas
  \[
    \sin(x+y) = \sin x\,\cos y + \cos x\,\sin y,
    \qquad
    \cos(x+y) = \cos x\,\cos y - \sin x\,\sin y
    \qquad(x,y\in\RR),
  \]
  and the fundamental inequality $0<\sin x<x<\tfrac{\sin x}{\cos x}$ for
  sufficiently small $x>0$, again deferring existence to the later theory of
  power series. For the remaining two functions, the two authors part ways.
  Apostol defines the logarithm by integral inversion,
  $\log x = \int_{1}^{x}\mathrm{d}t/t$. Shilov, by contrast, characterises the
  \emph{logarithm} axiomatically as the unique increasing function on
  $(0,+\infty)$ satisfying the homomorphism law $\ell(xy)=\ell(x)+\ell(y)$
  together with a normalisation $\ell(a)=1$ for a given $a>1$, and subsequently obtains
  the exponential as its inverse function.
\end{description}

The axiomatic characterisations just recalled share a structural feature that,
on reflection, one might have thought unavoidable: each defines its function
through a relation between its values at a \emph{pair} of independent points.
The addition and difference formulas for the trigonometric functions relate
$\cos(x\pm y)$ to the values at $x$ and at $y$. The exponential, in the standard
treatment, is fixed by the Cauchy relation $\je(x+y)=\je(x)\je(y)$, expressing
that it is a homomorphism of the additive group into the multiplicative one; the
logarithm, dually, is fixed by $\ell(xy)=\ell(x)+\ell(y)$. Which of the last two is taken
as primitive and which as the inverse of the other is a matter of expository
taste. Shilov starts from the logarithm, whereas texts such as
Prodi's~\cite{prodi1977} start from the exponential. However, in either case, the
defining equation is \emph{binary}: it constrains the function on
$\RR\times\RR$, and it is exactly this binary, homomorphism-theoretic form that
has long been taken to be the essence of an axiomatic definition. It is also
what obstructs a fixed-point reading, because a relation in two independent arguments is
not an operator acting on functions of one variable, leaving the
contraction principle with nothing to act upon.

The starting point of the present paper is the observation that the \emph{binary
relation is not necessary}. The duplication identities 
\begin{equation}
\je(2x)=\je(x)^{2}\quad \text{and}\quad \log(x^{2})=2\log x
\end{equation}
are \emph{unary}: each relates the value of the function
at a single point to its value at the doubled or squared argument, imposing no
constraint that involves a second, independent variable. Supplemented only by an
elementary normalisation at the origin and a minimal regularity requirement,
they suffice to determine the function uniquely. Consequently, the binary homomorphism law is
recovered as a \emph{consequence}, rather than assumed as a hypothesis. This
reduction from a binary to a unary relation is precisely what makes the
fixed-point formulation possible. A unary duplication identity naturally defines
an operator acting on a space of single-variable functions, and it is exactly upon this
operator that the Banach contraction principle can be applied. The remainder of the
paper develops this observation and, in Section~\ref{sec:computational}, details its
computational consequences.

Two further remarks locate our construction more sharply. First, the geometric
definition of the trigonometric functions through arc length on the unit circle
conceals a well-known circularity, noted by Richman~\cite{richman1984} and
revisited recently~\cite{babilio2025}: arc length is used to define sine and
cosine, yet those functions are later used to define arc length by
integration~\cite{stewart2016,swokowski2017}. A fixed-point characterisation is
free of this circularity because it never invokes arc length. Second, the
duplication structure we exploit is closely related to the ``scaling and squaring''
method for the matrix exponential~\cite{almohy2009new, higham2005scaling, higham2008, moler2003}, which is one of the most effective algorithms in numerical linear algebra. Furthermore, the study of the underlying functional equations is itself
classical~\cite{aczel1966,kuczma2009,granas2003,zeidler1986,trefethen2019}. Our contribution is to make the contraction \emph{global and explicit}, achieving uniformity on arbitrary compact intervals through a weighted norm, and thereby turning the characterisation into a numerically stable iteration.

\section{Analysis of the Fixed-Point Operator} \label{sec:fixpointOp}
In this section, we prove the existence and uniqueness of a continuous function
$\jetwo:\RR\to\CC$ satisfying the duplication and initial-condition axioms
$(\mathbf{A_1},\mathbf{A_2})$, defined in \eqref{eqax:A1} and
\eqref{eqax:A2}. The key ingredient of the proof is the classical
Banach--Caccioppoli contraction principle, which we recall for completeness
(see \cite{granas2003,zeidler1986} for further details).

\begin{theorem}[Banach--Caccioppoli]
  \label{thm:banach}Let $(X, d)$ be a complete metric space, and let $T : X
  \rightarrow X$ be a contraction mapping; that is, there exists a constant
  $\kappa \in [0, 1)$ such that
  \begin{equation}
    d (T (x), T (y)) \leqslant \kappa d (x, y)  \quad \text{for all } x, y \in
    X.
  \end{equation}
  Then $T$ admits a unique fixed point: there exists a unique $x^{\ast} \in X$
  such that $T (x^{\ast}) = x^{\ast}$. Furthermore, for any arbitrary initial
  point $x_0 \in X$, the iterative sequence defined by $x_{n + 1} = T (x_n)$
  converges to $x^{\ast}$ in the metric $d$, and the following error estimates
  hold for all $n \geqslant 1$:
  \begin{enumerate}
    \item A priori estimate:
    \begin{equation}
      d (x_n, x^{\ast}) \leqslant \frac{\kappa^n}{1 - \kappa} d (x_1, x_0) .
    \end{equation}
    \item A posteriori estimate:
    \begin{equation}
      d (x_n, x^{\ast}) \leqslant \frac{\kappa}{1 - \kappa} d (x_n, x_{n - 1})
      .
    \end{equation}
  \end{enumerate}
\end{theorem}

\begin{remark}
  The error estimates of Theorem~\ref{thm:banach} supply both theoretical and practical guarantees for the convergence behaviour of iterative schemes. The {{\em a priori\/}} estimate establishes a theoretical convergence rate: the
  error decays at least geometrically, with a factor $\kappa^n$, and thus
  quantifies how rapidly iterates approach the fixed point in the asymptotic
  regime. The {{\em a posteriori\/}} estimate, by contrast, is computable from
  the iterates and furnishes an on-line measure of the current error. In
  numerical implementations this distinction is fundamental: the {{\em a
  priori\/}} bound informs algorithm design and expectations about worst-case
  performance, while the {{\em a posteriori\/}} bound is indispensable for
  reliable error control, stopping criteria, and the detection of spurious
  stagnation. In particular, the latter guarantees that if the inter-iterate
  distance $d (x_n, x_{n - 1})$ is small, then the iterate $x_n$ is provably
  close to the true solution $x^{\ast}$, so that apparent stagnation cannot
  occur at a significant distance from the fixed point.
\end{remark}

Our argument relies on the following two lemmas, whose proofs (which utilize
Theorem~\ref{thm:banach} operating on suitable weighted norms) are deferred to
the next section.

\begin{lemma}
  \label{lemma:fixpoints}Let $\gamma \in \left\{ 1, \ji \right\}$. There
  exists a unique {{\em continuous\/}} function $\jqtwo_{\ast} \in
  \mathcal{C}^0 \left( \RR, \CC \right)$ which solves the functional equation
  \begin{equation}
    \mathfrak{q} (x) = \frac{\gamma^2}{4} + \left( \frac{1}{2} + \gamma
    \frac{x}{4} \right) \mathfrak{q} \left( \frac{x}{2} \right) +
    \frac{x^2}{16} \mathfrak{q}^2 \left( \frac{x}{2} \right) \qquad \left( x
    \in \RR \right) . \label{eq:functional_q}
  \end{equation}
  Moreover, $\jqtwo_{\ast}$ is the unique solution to {{\em
  \eqref{eq:functional_q}\/}} on any interval $I_a := [- a, a]$, $a > 0$,
  within the class $\mathcal{B} (I_a, \CC)$ of {{\em bounded\/}} (not
  necessarily continuous) functions.
\end{lemma}

\begin{lemma}
  \label{lemma:ifexistse}Let $\gamma \in \left\{ 1, \ji \right\}$. Let
  $\je: \RR \to \CC$ be a continuous function satisfying the
  duplication and initial condition axioms $(\mathbf{A_1}, \mathbf{A_2})$.
  Then, there exists a unique continuous function $\mathfrak{q}: \RR \to \CC$,
  which we refer to as the residue, such that
  \begin{equation}
    \je (x) = 1 + \gamma x + x^2 \mathfrak{q} (x)  \qquad (x \in \RR)
    . \label{eq:ansatzq}
  \end{equation}
  Moreover, this function $\mathfrak{q}$ coincides with the unique continuous
  solution of the functional equation {{\em \eqref{eq:functional_q}\/}}. In
  particular, $\mathfrak{q} (0) = \gamma^2 / 2$.
\end{lemma}

\begin{remark}
  We emphasize that the uniqueness of $\jq$ must be understood within the
  class of continuous functions. Without this restriction, the value of $\jq
  (0)$ could be chosen arbitrarily, since the factor $x^2$ in
  \eqref{eq:ansatzq} makes the representation insensitive to the value at the
  origin. Continuity therefore ensures that the residue is uniquely
  determined.
\end{remark}

\begin{theorem}
  \label{thm:main}There exists a unique continuous function $\je: \RR
  \to \CC$ satisfying the duplication and initial condition axioms
  $(\mathbf{A_1}, \mathbf{A_2})$. Specifically, $\je (x) = 1 + \gamma
  x + x^2 \mathfrak{q} (x)$ with $\mathfrak{q}$ being the unique continuous
  solution of the functional equation {{\em \eqref{eq:functional_q}\/}}.
\end{theorem}

\begin{proof}
  We treat uniqueness and existence separately.{\smallskip}
  
  {\noindent}{{\em Uniqueness.\/}} Suppose that $\je_1$ and
  $\je_2$ are two continuous functions satisfying both axioms
  $(\mathbf{A_1}, \mathbf{A_2})$. By Lemma~\ref{lemma:ifexistse}, there exist
  continuous residue functions $\mathfrak{q}_1$ and $\mathfrak{q}_2$ such that
  $\je_j (x) = 1 + \gamma x + x^2 \mathfrak{q}_j (x)$ for $j \in \{1,
  2\}$, and both residues obey the functional equation
  \eqref{eq:functional_q}. By Lemma~\ref{lemma:fixpoints}, this functional
  equation has a unique continuous solution; hence $\mathfrak{q}_1
  =\mathfrak{q}_2$. It follows that $\je_1 (x) =\je_2 (x)$
  for all $x \in \RR$, proving uniqueness.{\smallskip}
  
  {\noindent}{{\em Existence.\/}} Let ${\color{black}{\mathfrak{q}_{\ast}}}$
  denote the unique continuous solution of \eqref{eq:functional_q} provided by
  Lemma~\ref{lemma:fixpoints}, and define
  \begin{equation}
    {\je} (x) := 1 + \gamma x + x^2 
    {\color{black}{\mathfrak{q}}}_{\ast} (x) .
  \end{equation}
  We verify that ${\je}$ satisfies both axioms.
  
  To check the duplication identity $(\mathbf{A_1})$, we compute the square of
  $\je (x)$:
  \begin{eqnarray}
    \je^2 (x) & = & (1 + \gamma x + x^2 \mathfrak{q}_{\ast} (x))^2
    \nonumber\\
    & = & 1 + \gamma^2 x^2 + x^4 \mathfrak{q}_{\ast}^2 (x) + 2 \gamma x + 2
    x^2 \mathfrak{q}_{\ast} (x) + 2 \gamma x^3 \mathfrak{q}_{\ast} (x)
    \nonumber\\
    & = & 1 + \gamma (2 x) + (2 x)^2  \left[ \frac{\gamma^2}{4} + \left(
    \frac{1}{2} + \gamma \frac{2 x}{4} \right) \mathfrak{q}_{\ast} (x) +
    \frac{(2 x)^2}{16} \mathfrak{q}_{\ast}^2 (x) \right] . 
  \end{eqnarray}
  Since $\mathfrak{q}_{\ast}$ satisfies the functional equation
  \eqref{eq:functional_q}, the expression inside the square brackets is
  precisely the definition of ${\color{black}{\mathfrak{q}_{\ast} (2 x)}}$.
  Therefore, the equation simplifies to:
  \begin{equation}
    {\je}^2 (x) = 1 + \gamma (2 x) + (2 x)^2 
    {\color{black}{\mathfrak{q}}}_{\ast}  (2 x) =
    {\je} (2 x),
  \end{equation}
  which proves that axiom $(\mathbf{A_1})$ holds.
  
  Finally, we verify the initial conditions $(\mathbf{A_2})$. Direct
  evaluation yields $\je (0) = 1$. For the derivative at the origin,
  we compute the limit of the difference quotient:
  \begin{equation}
    {\je}' (0) = \lim_{x \to 0} 
    \frac{{\je} (x) - {\je}
    (0)}{x} = \lim_{x \to 0}  \frac{1 + \gamma x + x^2 
    {\color{black}{\mathfrak{q}}}_{\ast} (x) - 1}{x} = \lim_{x \to 0} (\gamma
    + x {\color{black}{\mathfrak{q}}}_{\ast} (x)) .
  \end{equation}
  Since $\mathfrak{q}_{\ast}$ is bounded (and continuous) at $0$, we have
  $\lim_{x \to 0} x\mathfrak{q}_{\ast} (x) = 0$; thus $\je' (0) =
  \gamma$. This confirms $(\mathbf{A_2})$ and completes the proof.
\end{proof}

\section{Proofs of Lemma~\ref{lemma:fixpoints} and \ref{lemma:ifexistse}} \label{sec:proofs}

The existence and uniqueness of the fixed point on an arbitrary compact
interval $I_a := [- a, a]$ with $a > 0$, are established through the
construction of a weighted Banach space. This setting enables effective
control of the contraction induced by the linear component of the operator,
while simultaneously suppressing the quadratic nonlinearity and ensuring a
uniform bound on the parameter $\gamma$.

\subsection{Admissible weight sequences}A central difficulty in proving
existence and uniqueness lies in handling the nonlinearity of the operator
$T$. When $\mathcal{C}^0$ is equipped with the standard uniform norm, the
quadratic term prevents $T$ from being a {{\em global\/}} contraction. To
address this issue, we introduce the notion of admissible weight sequences. By
working within a suitably weighted Banach space, one can control the linear
contraction of the domain and suppress the quadratic growth by allowing a
tunable parameter $\lambda$ to tend to infinity.

In this subsection, we first formulate a set of abstract axioms that
characterize these admissible weight sequences. Next, we demonstrate the
practical applicability of this framework by explicitly verifying the
conditions for two concrete families of functions: Gaussian weights and
polynomial power weights.

\begin{definition}[Admissible weight sequences]
  \label{eq:defAWS}Let $\Omega = (\omega_{\lambda})_{\lambda \in \NN}$ be a
  sequence of continuous functions $\omega_{\lambda} : \RR \to \RR$. We say
  $\Omega$ is an \text{{\itshape{admissible weight sequence}}} if, for any
  compact interval $I_a$, it satisfies the following properties eventually in
  $\lambda$, in the sense that for every compact interval $I_a$ there exists
  $\lambda_{\ast} = \lambda_{\ast} (a) \in \NN$ such that for all $\lambda
  \geqslant \lambda_{\ast}$ the following conditions hold:{\smallskip}
  \begin{enumerate}
    \item {{\em Strict positivity and lower bound\/}}:
    \begin{equation}
      \inf_{x \in I_a} \omega_{\lambda} (x) \geqslant 1.
      \label{eq:omegalowbound}
    \end{equation}
    \item {{\em Linear suppression\/}}: Defining the ratio
    \begin{equation}
      R_1 (x, \lambda) := \frac{\omega_{\lambda}  (x / 2)}{\omega_{\lambda}
      (x)}, \label{eq:R1def}
    \end{equation}
    we have
    \begin{equation}
      \sup_{x \in I_a} R_1 (x, \lambda) \leqslant 1 \quad \text{and} \quad
      \lim_{\lambda \to \infty} \sup_{x \in I_a} [|x|R_1 (x, \lambda)] = 0.
      \label{eq:R1prop}
    \end{equation}
    \item {{\em Quadratic suppression\/}}: Defining the {{\em quadratic\/}}
    ratio
    \begin{equation}
      R_2 (x, \lambda) := \frac{\omega_{\lambda}^2  (x / 2)}{\omega_{\lambda}
      (x)}, \label{eq:R2def}
    \end{equation}
    we require
    \begin{equation}
      \lim_{\lambda \to \infty} \sup_{x \in I_a} [x^2 R_2 (x, \lambda)] = 0.
      \label{eq:R2prop}
    \end{equation}
  \end{enumerate}
  We emphasize that all estimates are uniform on $I_a$. However, the threshold
  $\lambda_{\ast}$ is allowed to depend on the interval size $a$.
\end{definition}

Note that all estimates are uniform on $I_a$, but the threshold
$\lambda_{\ast}$ and the associated constants may depend on $a$. Let
$\mathcal{X}_a = \left( \mathcal{C}^0 (I_a, \CC), \left\| \cdot \,
\right\|_{\lambda} \right)$ be the normed space of continuous functions
equipped with the weighted norm (note that $\mathcal{X}_a$ also depends on
$\lambda$, though we suppress this in the notation for brevity):
\begin{equation}
  \| \jq \|_{\lambda} := \sup_{x \in I_a}  \frac{| \jq (x) |}{\omega_{\lambda}
  (x)} .
\end{equation}
Since $\omega_{\lambda}$ is continuous and strictly positive on the compact
set $I_a$, there exist constants $0 < m_{\lambda} \leqslant M_{\lambda} <
\infty$ such that $m_{\lambda} \leqslant \omega_{\lambda} (x) \leqslant
M_{\lambda}$ for all $x \in I_a$. It follows that
\begin{equation}
  \frac{1}{M_{\lambda}} \|f\|_{\infty} \leqslant \|f\|_{\lambda} \leqslant
  \frac{1}{m_{\lambda}} \|f\|_{\infty},
\end{equation}
so the weighted norm $\| \cdot \|_{\lambda}$ is equivalent to the standard
uniform norm on $\mathcal{C}^0 (I_a, \CC)$, and hence $\mathcal{X}_a$ is a
Banach space.

Let $\jball = \{ \jq \in \mathcal{X}_a : \| \jq \|_{\lambda} \leqslant 1\}$
denote the closed unit ball. Endowed with the metric induced by $\| \cdot
\|_{\lambda}$, the set $\jball$ is a complete metric space, being a closed
subset of a Banach space.

Our goal is to show that, for sufficiently large $\lambda$, the operator $T$
maps $\jball$ into itself and is a strict contraction on $\jball$. The
Banach--Caccioppoli fixed-point theorem (Theorem~\ref{thm:banach}) will then
guarantee the existence and uniqueness of a fixed point of $T$ in $\jball$.

Before that, to illustrate the effectiveness of this framework, we now verify
the admissibility conditions for two concrete classes of weights: Gaussian
weights and polynomial power weights. In both cases, the key estimates rely on
the following elementary bounds: \ for any $c > 0$ and $\lambda \in \NN$, the
functions $|x| e^{- c \lambda x^2}$ and $x^2 e^{- c \lambda x^2}$ attain their
global maxima on $\RR$ at $x^2 = 1 / (2 c \lambda)$ and $x^2 = 1 / (c
\lambda)$, respectively. Consequently, we have the uniform global bounds:
\begin{equation}
  \sup_{x \in \RR} [|x|e^{- c \lambda x^2}] = \frac{1}{\sqrt{2 c  e \lambda}} 
  \quad \text{and} \quad \sup_{x \in \RR} [x^2 e^{- c \lambda x^2}] =
  \frac{1}{c  e  \lambda} \label{eq:exponential_suppression_bounds}
\end{equation}
In particular, both quantities vanish as $\lambda \to \infty$.

\begin{proposition}[Gaussian Weights]
  The sequence of Gaussian weights
  \begin{equation}
    \omega_{\lambda} (x) = e^{\lambda x^2}, \quad \lambda \in \NN,
  \end{equation}
  is an admissible weight sequence.
\end{proposition}

\begin{proof}
  Let $I_a = [- a, a]$ be an arbitrary compact interval. We verify the three
  conditions of Definition~\ref{eq:defAWS}. The {{\em lower bound\/}}
  \eqref{eq:omegalowbound} is clear: Since $\lambda \geqslant 0$ we trivially
  have $\omega_{\lambda} (x) = e^{\lambda x^2} \geqslant e^0 = 1$ for all $x
  \in I_a$. For the {{\em linear suppression\/}}, we observe that $R_1 (x,
  \lambda) = e^{- (3 / 4) \lambda x^2} \leqslant 1$. Hence, applying the bound
  from \eqref{eq:exponential_suppression_bounds} with $c = 3 / 4$, we obtain:
  \[ \sup_{x \in I_a} [|x|R_1 (x, \lambda)] \leqslant \frac{1}{\sqrt{(3 / 2) e
     \lambda}} \xrightarrow{\lambda \to \infty} 0. \]
  It remains to prove the {{\em quadratic suppression\/}} condition
  \eqref{eq:R2prop}. The quadratic ratio is $R_2 (x, \lambda) = e^{- (1 / 2)
  \lambda x^2}$. Applying \eqref{eq:exponential_suppression_bounds} with $c =
  1 / 2$:
  \[ \sup_{x \in I_a} [x^2 R_2 (x, \lambda)] \leqslant \frac{2}{e \lambda}
     \xrightarrow{\lambda \to \infty} 0. \]
  Thus, the Gaussian sequence is admissible.
\end{proof}

\begin{remark}
  While the family of Gaussian weights has the distinct advantage of being
  universally admissible independently of the interval size $a$, it can seem
  conceptually awkward to use the exponential function to prove the existence
  of the exponential function. Although this approach is formally and
  logically sound, from a didactic perspective it is highly desirable to
  construct admissible sequences that do not rely on the exponential function
  at all. This ensures that our framework can serve as a rigorous,
  self-contained primary definition of the exponential function in real and
  complex analysis. A simple and elegant alternative is the family of
  polynomial power weights. Note, however, that these weights are locally
  tailored, meaning the choice of the sequence now explicitly depends on $a$.
\end{remark}

\begin{proposition}[Polynomial Power Weights]
  For any fixed $a > 0$, choose a constant $c$ such that $0 < c < 8 / a^2$.
  Then the sequence of polynomial power weights
  \begin{equation}
    \omega_{\lambda} (x) = (1 + c  x^2)^{\lambda}, \quad \lambda \in \NN,
  \end{equation}
  is an admissible weight sequence on the interval $I_a = [- a, a]$.
\end{proposition}

\begin{proof}
  Let $I_a = [- a, a]$ be an arbitrary compact interval. We verify the three
  conditions of Definition~\ref{eq:defAWS}. The lower bound
  \eqref{eq:omegalowbound} is trivial: since $c > 0$, we have
  $\omega_{\lambda} (x) \geqslant 1$ for all $x \in I_a$. For the {{\em linear
  suppression\/}}, the ratio is $R_1 (x, \lambda) = (C (x^2))^{\lambda}$,
  where for every $t \geqslant 0$ the function $C$ is given by:
  \[ C (t) = \frac{1 + c  t / 4}{1 + c  t} = 1 - \frac{3 c}{4 (1 + c  t)} t.
  \]
  Clearly $C (0) = 1$. For $t \in [0, a^2]$, we can bound the denominator from
  above by $1 + c  a^2$. Therefore:
  \[ C (t) \leqslant 1 - \alpha t, \quad \text{where} \quad \alpha := \frac{3
     c}{4 (1 + c  a^2)} > 0. \]
  Using the standard inequality $1 - s \leqslant e^{- s}$ for $s \geqslant 0$,
  we obtain $C (t) \leqslant e^{- \alpha t}$, and thus $R_1 (x, \lambda)
  \leqslant e^{- \alpha \lambda x^2} \leqslant 1$. Applying the exponential
  bound \eqref{eq:exponential_suppression_bounds} with the constant $\alpha$,
  the linear suppression holds:
  \[ \sup_{x \in I_a} [|x|R_1 (x, \lambda)] \leqslant \sup_{x \in I_a}
     [|x|e^{- \alpha \lambda x^2}] \leqslant \frac{1}{\sqrt{2 \alpha e
     \lambda}} \xrightarrow{\lambda \to \infty} 0. \]
  For the {{\em quadratic suppression\/}} \eqref{eq:R2prop}, the ratio
  simplifies to $R_2 (x, \lambda) = (B (x^2))^{\lambda}$, with the base:
  \[ B (t) = \frac{(1 + c  t / 4)^2}{1 + c  t} = 1 - \frac{c  t (8 - c  t)}{16
     (1 + c  t)} . \]
  By our strict choice of $c < 8 / a^2$, the term $(8 - c  t)$ is strictly
  bounded from below by $(8 - c  a^2) > 0$ for all $t \in [0, a^2]$. Bounding
  the denominator from above by $16 (1 + c  a^2)$, we obtain:
  \[ B (t) \leqslant 1 - \beta t, \quad \text{where} \quad \beta := \frac{c (8
     - c  a^2)}{16 (1 + c  a^2)} > 0. \]
  Using $1 - s \leqslant e^{- s}$ again, $B (t) \leqslant e^{- \beta t}$,
  which gives $R_2 (x, \lambda) \leqslant e^{- \beta \lambda x^2}$. Applying
  \eqref{eq:exponential_suppression_bounds} with the constant $\beta$ yields:
  \[ \sup_{x \in I_a} [x^2 R_2 (x, \lambda)] \leqslant \sup_{x \in I_a} [x^2
     e^{- \beta \lambda x^2}] \leqslant \frac{1}{\beta e \lambda}
     \xrightarrow{\lambda \to \infty} 0. \]
  Thus, the polynomial power sequence is locally admissible.
\end{proof}

\subsection{Proof of Lemma~\ref{lemma:fixpoints}}We define the operator $T$
acting on functions $\jq : \RR \to \CC$ by
\begin{equation}
  (T \jq) (x) = \frac{\gamma^2}{4} + \left( \frac{1}{2} + \gamma \frac{x}{4}
  \right) \jq \left( \frac{x}{2} \right) + \frac{x^2}{16} \jq^2 \left(
  \frac{x}{2} \right), \label{eq:operator_T_unified}
\end{equation}
so that the proof reduces to the existence of $\jq$ satisfying
\eqref{eq:functional_q} to the fixed point problem $T \jq = \jq$.

To prove uniqueness and existence of the fixed point, we verify that for every
$a > 0$ there exists a sufficiently large $\lambda > 0$ such that the
hypotheses of the Banach--Caccioppoli contraction principle
(Theorem~\ref{thm:banach}) apply to the Banach space $\bigl( \mathcal{C}^0
(I_a, \CC), \| \cdot \|_{\lambda} \bigr)$, yielding the existence and
uniqueness of a fixed point of $T$ in $\jball$. We proceed in three steps.

\subsubsection*{Step 1: Self-Mapping {\opt}$T (\jball) \subseteq
\jball${\cpt}}We verify that for any admissible weight sequence, choosing
$\lambda$ sufficiently large ensures $T$ maps the unit ball into itself. Let
$\jq \in \jball$, which implies $| \jq (x) | \leqslant \omega_{\lambda} (x)$.
Noting that $| \gamma |^2 = 1$ because $\gamma \in \{1, \ji \}$, we apply the
triangle inequality to \eqref{eq:operator_T_unified} to estimate:
\[ | (T \jq) (x) | \leqslant \frac{1}{4} + \left( \frac{1}{2} + | \gamma |
   \frac{|x|}{4} \right) \omega_{\lambda} \left( \frac{x}{2} \right) +
   \frac{x^2}{16} \omega_{\lambda} \left( \frac{x}{2} \right)^2\,. \]
Dividing by $\omega_{\lambda} (x)$ to switch to the weighted norm, we obtain:
\[ \frac{| (T \jq) (x) |}{\omega_{\lambda} (x)} \leqslant \frac{1}{4
   \omega_{\lambda} (x)} + \left( \frac{1}{2} + | \gamma | \frac{|x|}{4}
   \right) R_1 (x, \lambda) + \frac{x^2}{16} R_2 (x, \lambda) \]
with $R_1, R_2$ defined by \eqref{eq:R1def}, \eqref{eq:R2def}. Taking the
supremum over $I_a$ and applying the admissibility properties
\eqref{eq:omegalowbound} and \eqref{eq:R1prop} (i.e., $\omega_{\lambda}
\geqslant 1$ and $R_1 \leqslant 1$):
\[ \|T \jq \|_{\lambda} \leqslant \frac{1}{4} + \frac{1}{2} + \frac{| \gamma
   |}{4} \sup_{x \in I_a} [|x|R_1 (x, \lambda)] + \frac{1}{16} \sup_{x \in
   I_a} [x^2 R_2 (x, \lambda)]\,. \]
As $\lambda \to \infty$, the last two terms vanish by the properties of the
admissible weights, bounding the norm by $3 / 4 + o(1) < 1$ for $\lambda$ sufficiently large. Thus, $T (\jball) \subseteq \jball$ for such values of $\lambda$.

\subsubsection*{Step 2: Strict Contraction on Compact Intervals}Let $\jq_1,
\jq_2 \in \jball$. We evaluate the difference:
\[ (T \jq_1) (x) - (T \jq_2) (x) = \left[ \left( \frac{1}{2} + \gamma
   \frac{x}{4} \right) + \frac{x^2}{16}  \left( \jq_1 \left( \frac{x}{2}
   \right) + \jq_2 \left( \frac{x}{2} \right) \right) \right]  \left[ \jq_1
   \left( \frac{x}{2} \right) - \jq_2 \left( \frac{x}{2} \right) \right] . \]
Since both functions belong to $\jball$, we bound the sum by $| \jq_1 (x / 2)
+ \jq_2 (x / 2) | \leqslant 2 \omega_{\lambda}  (x / 2)$. Extracting the norm
of the difference, we have $| \jq_1 (x / 2) - \jq_2 (x / 2) | \leqslant
\omega_{\lambda}  (x / 2)  \| \jq_1 - \jq_2 \|_{\lambda}$. Taking the absolute
value and dividing the entire expression by $\omega_{\lambda} (x)$ yields:
\[ \frac{| (T \jq_1) (x) - (T \jq_2) (x) |}{\omega_{\lambda} (x)} \leqslant
   \left[ \left( \frac{1}{2} + | \gamma | \frac{|x|}{4} \right) R_1 (x,
   \lambda) + \frac{x^2}{8} R_2 (x, \lambda) \right]  \| \jq_1 - \jq_2
   \|_{\lambda} \]
Taking the supremum over $I_a$, we obtain the Lipschitz constant $\kappa
(\lambda)$:
\[ \kappa (\lambda) = \sup_{x \in I_a} \left[ \frac{1}{2} R_1 (x, \lambda) +
   \frac{| \gamma |}{4} |x|R_1 (x, \lambda) + \frac{x^2}{8} R_2 (x, \lambda)
   \right]\,. \]
Because $R_1 (x, \lambda) \leqslant 1$ (see \eqref{eq:R1prop}), the first term
is strictly bounded by $1 / 2$. By the admissibility properties of the weight,
the second and third terms vanish uniformly on $I_a$ as $\lambda \to \infty$.
Consequently:
\[ \limsup_{\lambda \to \infty} \kappa (\lambda) \leqslant \frac{1}{2} < 1\,. \]
Therefore, for sufficiently large $\lambda$, $T$ is a strict contraction on
$\jball$. By the Banach-Caccioppoli theorem, there exists a unique fixed point
$\jq_a \in \jball$.

\subsubsection*{Step 3: Global Consistency and Uniqueness}We have proven that
for any $a > 0$, there exists a unique continuous solution $\jq_a$ on $I_a$.
If $0 < a_1 < a_2$, the restriction of $\jq_{a_2}$ to $I_{a_1}$ must be a
solution to the fixed point problem on $I_{a_1}$. By uniqueness, this
restriction must coincide with $\jq_{a_1}$. This compatibility allows us to
define a unique global function $\jq_{\ast} : \RR \to \CC$ by setting
$\jq_{\ast} (x) = \jq_a (x)$ for any $x \in I_a$. Furthermore, the same
contraction argument holds if we replace $\mathcal{C}^0 (I_a, \CC)$ with the
Banach space of \text{{\itshape{bounded}}} functions $\mathcal{B} (I_a, \CC)$.
Thus, $\jq_{\ast}$ is the unique solution even among bounded
functions.{\hspace*{\fill}}$\Box$

\subsection{Proof of Lemma~\ref{lemma:ifexistse}}In this subsection we prove
Lemma~\ref{lemma:ifexistse} whose utility is in showing that there is a one to
one correspondence between the set of quadratic residuals and the
corresponding ``exponential'' functions $\je$ satisfying axioms
$(\mathbf{A_1}, \mathbf{A_2})$. Note that, at this stage, we still don't know
that a function $\je$ satisfying these axioms exists, nor that it is unique.

\begin{proof}
  For $x \neq 0$, the function $\jqtwo$ is uniquely determined by inverting
  \eqref{eq:ansatzq}:
  \begin{equation}
    \jqtwo = \frac{\jetwo (x) - (1 + \gamma x)}{x^2} . \label{eq:ansatzqfrome}
  \end{equation}
  Since $\jetwo$ is continuous, $\jqtwo$ is continuous on $\RR \setminus
  \{0\}$. To ensure continuity on the whole of $\RR$, we must verify that the
  limit exists as $x \to 0$ and equals $\jqtwo (0)$. We proceed in three
  steps.{\smallskip}
  
  {\noindent}{{\em Step 1: Functional Equation\/}}. We first verify that
  $\jqtwo$ satisfies~\eqref{eq:functional_q} for $x \neq 0$. Substituting the
  expression \eqref{eq:ansatzq} into the duplication identity
  $\jetwo (x) = \jetwo^2 (x / 2)$, we
  obtain:
  \begin{equation}
    1 + \gamma x + x^2 \jq (x) = \left( 1 + \gamma \frac{x}{2} + \frac{x^2}{4}
    \jq \left( \frac{x}{2} \right) \right)^2 .
  \end{equation}
  Expanding the square on the right-hand side, cancelling the linear term $1 +
  \gamma x$, and dividing by $x^2$ yields exactly the functional equation
  \eqref{eq:functional_q}. We define $\jqtwo (0) = \gamma^2 / 2$ so that the
  equation holds at $x = 0$ as well. However, we must rule out the possibility
  that $\jqtwo$ is unbounded near the origin.{\smallskip}
  
  {\noindent}{{\em Step 2: Boundedness near the origin\/}}. We show that any
  residue $\jq : \RR \to \CC$ satisfying the ansatz must be bounded around the
  origin. From the differentiability of $\je$ at 0 (Axiom
  $\mathbf{A_2}$), we deduce that $x \jq (x) \to 0$ as $x \to 0$.
  Consequently, there exists a constant $c > 0$ and an interval $I_{\delta} :=
  (- \delta, \delta)$ with $\delta > 0$, such that $|x \jq (x) | \leqslant c$
  for every $x \in I_{\delta}$.
  
  Using the functional equation \eqref{eq:functional_q}, we can bound $\left|
  \jqtwo \right|$ as follows
  \[ \left| \jq (x) \right| \leqslant \frac{1}{2} \left| \jq \left(
     \frac{x}{2} \right) \right| + \frac{1}{4}  \left( | \gamma | + \left|
     \frac{x}{2} \jq \left( \frac{x}{2} \right) \right| \right)^2 \leqslant
     \frac{1}{2} \left| \jq \left( \frac{x}{2} \right) \right| + \frac{1}{4} 
     (| \gamma | + c)^2 . \]
  We set $\kappa := \frac{1}{4} (| \gamma | + c)^2$ and iterate the previous
  inequality $j$ times to yield, at least for every $x \neq 0$,
  \[ \left| \jqtwo (x) \right| \leqslant \frac{1}{2^j} \left| \jq \left(
     \frac{x}{2^j} \right) \right| + \kappa \sum_{i = 0}^{j - 1} \frac{1}{2^i}
     = \frac{1}{|x|} \left| \frac{x}{2^j} \jq \left( \frac{x}{2^j} \right)
     \right| + \kappa \sum_{i = 0}^{j - 1} \frac{1}{2^i} . \]
  Taking the limit as $j \rightarrow \infty$, the first term vanishes (because
  $a_j \jqtwo (a_j) \rightarrow 0$ for any sequence $(a_j)_{j \in \NN}$ that
  converges to zero), and the geometric series converges to $2 \kappa$. Thus
  $\left| \jqtwo (x) \right| \leqslant 2 \kappa$ for every $x \in I_{\delta}$,
  proving that $\jq$ is uniformly bounded near $0$.{\smallskip}
  
  {\noindent}{{\em Step 3: Uniqueness of the residue.\/}} So far, we have
  shown that $\jqtwo$ is a global solution of the functional equation
  \eqref{eq:functional_q} that is continuous on $\RR \setminus \{ 0 \}$ and
  bounded on $I_{\delta}$. By Lemma~\ref{lemma:fixpoints}, there exists a
  {{\em unique\/}} bounded solution $\jqtwo_{\ast}$ of the functional
  equation\eqref{eq:functional_q}, which is known to be continuous everywhere.
  Therefore $\jqtwo = \jqtwo_{\ast}$ on $I_{\delta}$. This implies that
  $\jqtwo$ is continuous at the origin as well.
\end{proof}

\section{Fixed point characterization of the Natural
Logarithm}\label{sec:banach_proof}

In this section, we prove that the natural logarithm is the unique fixed point
of a suitably defined linear operator acting on a carefully constructed complete metric space.

\subsection{Definition of the Operator and the Functional Setup}Let $I = [1 /
A, A]$ be a closed interval with $A > 1$. We define the linear functional
operator $T$ as:
\begin{equation}
  T [f] (x) = 2 f (\sqrt{x})\,.
\end{equation}
The transformation $x \mapsto \sqrt{x}$ maps the interval $I$ strictly into
itself, since $\sqrt{x} \in [1 / \sqrt{A}, \sqrt{A}] \subset I$ for all
$x \in I$. As a consequence, all iterates of $T$ remain within the domain $I$,
and the operator is well-defined on any subspace of $\mathcal{C}_b \left( I,
\RR \right)$, the Banach space of bounded continuous functions on $I$. Also,
the functional equation $T [f] = f$ naturally admits a one-parameter family of
solutions, $f (x) = c \log x$ for $c \in \RR$. 

To isolate the natural logarithm, corresponding to the normalization $c=1$, we
impose the initial conditions $f(1)=0$ and $f'(1)=1$ (cf.~axioms
$(\mathbf{B_1},\mathbf{B_2})$ in \eqref{eqax:B1} and \eqref{eqax:B2}). Within
our topological framework, these conditions are incorporated by introducing the
affine reference function $f_0(x)=x-1$ and by working in a suitably weighted
function space.

We now introduce the functional framework in which the fixed point analysis
will be carried out. Let $\mathcal{B} \left( I, \RR \right)$ denote the
subspace of $\mathcal{C}_b \left( I, \RR \right)$ consisting of all continuous
functions $h : I \to \RR$ such that $h (1) = 0$ and
\begin{equation}
  \|h\|_{\mathcal{B}} := \sup_{x \in I \setminus \{1\}}  \frac{|h (x) | \cdot
  x}{(x - 1)^2} < + \infty\,.
\end{equation}
This weighted norm captures the quadratic vanishing of functions at $x = 1$
and will play a crucial role in controlling the action of the operator $T$.

We also need the space $\mathcal{M} \left( I, \RR \right)$ as the affine
translation of $\mathcal{B} \left( I, \RR \right)$ by the reference function
$f_0 (x) = x - 1$:
\begin{equation}
  \mathcal{M} \left( I, \RR \right) := \left\{ f : I \to \RR \, : f - f_0 \in
  \mathcal{B} \right\} .
\end{equation}
\begin{lemma}
  The space $(\mathcal{B} \left( I, \RR \right), \| \cdot \|_{\mathcal{B}})$
  is a Banach space. Therefore its affine translation $\mathcal{M} \left( I,
  \RR \right) = f_0 +\mathcal{B} \left( I, \RR \right)$ is a closed subset of
  the ambient functional space and is therefore a complete metric space under
  the induced metric $d (f_1, f_2) = \|f_1 - f_2 \|_{\mathcal{B}}$.
\end{lemma}

\begin{proof}
  It is clear that $\mathcal{B} \left( I, \RR \right)$ is a normed space. It
  remains to prove completeness. Let $(h_n)_{n \in \NN}$ be a Cauchy sequence
  in $\mathcal{B} \left( I, \RR \right)$. We need to show the existence of an
  element $h \in \mathcal{B} \left( I, \RR \right)$ such that $h_n \rightarrow
  h$ in $\mathcal{B} \left( I, \RR \right)$.
  
  For each $n \in \NN$, define the auxiliary function $H_n (x) := h_n (x)
  \frac{x}{(x - 1)^2}$ on the punctured compact interval $I \setminus \{1\}$.
  By the definition of the norm on $\mathcal{B} \left( I, \RR \right)$, we
  have for any $n, m \in \NN$:
  \begin{equation}
    \|h_n - h_m \|_{\mathcal{B}} = \sup_{x \in I \setminus \{1\}} |H_n (x) -
    H_m (x) | = \|H_n - H_m \|_{\infty} .
  \end{equation}
  This isometry implies that $(H_n)_{n \in \NN}$ is a Cauchy sequence in the
  Banach space $\mathcal{C}_b \left( I\backslash \{ 1 \}, \RR \right)$ of
  bounded, continuous functions on $I \setminus \{1\}$, endowed with the
  supremum norm. Consequently, $H_n$ converges uniformly to a bounded and
  continuous function $H$ (in particular, its supremum norm is finite:
  $\|H\|_{\infty} < \infty$). We now define a candidate limit function $h : I
  \to \RR$ by
  \begin{equation}
    h (x) := \left\{\begin{array}{ll}
      H (x) \hspace{0.17em} \dfrac{(x - 1)^2}{x}, & x \neq 1,\\
      0, & x = 1.
    \end{array}\right.
  \end{equation}
  Since $x \in I = [1 / A, A]$ implies $x > 0$, this definition is meaningful.
  Moreover, $h$ is continuous on $I \setminus \{1\}$ because $H$ is continuous
  there. To check continuity at $x = 1$, note that
  \begin{equation}
    |h (x) - h (1) | = |h (x) | = |H (x) | \frac{(x - 1)^2}{x} \leqslant
    \|H\|_{\infty} \frac{(x - 1)^2}{x} \; \longrightarrow 0 \qquad \text{as }
    x \to 1.
  \end{equation}
  Thus, $h \in \mathcal{C}_b (I, \RR)$. Next, we verify that $h \in
  \mathcal{B} \left( I, \RR \right)$. We have already established $h (1) = 0$.
  Furthermore, evaluating its norm yields:
  \begin{equation}
    \|h\|_{\mathcal{B}} = \sup_{x \in I \setminus \{1\}}  \frac{|h (x) | \cdot
    x}{(x - 1)^2} = \sup_{x \in I \setminus \{1\}} |H (x) | = \|H\|_{\infty} <
    \infty\,.
  \end{equation}
  So $h \in \mathcal{B} \left( I, \RR \right)$. Moreover, $\|h_n -
  h\|_{\mathcal{B}} = \|H_n - H\|_{\infty} \longrightarrow 0$ because $H_n \to
  H$ uniformly. This proves that every Cauchy sequence in $\mathcal{B} \left(
  I, \RR \right)$ converges, and thus $(\mathcal{B} \left( I, \RR \right), \|
  \cdot \|_{\mathcal{B}})$ is a Banach space.
\end{proof}

\begin{theorem}[The natural logarithm]
  Let $I = [1 / A, A]$ be a closed interval with an arbitrary constant $A >
  1$. Then, the linear operator $T [f] (x) = 2 f (\sqrt{x})$ is a strict
  global contraction on the complete metric space $\mathcal{M} \left( I, \RR
  \right)$ with contraction constant $\kappa = 1 / 2$.
  
  Consequently, $T$ admits a unique fixed point in $\mathcal{M} \left( I, \RR
  \right)$, and this fixed point is the natural logarithm.
\end{theorem}

\begin{remark}
  The fact that the natural logarithm $\log x$ is a fixed point is immediate:
  $T [\log] (x) = 2 \log (\sqrt{x}) = \log x$. It remains only to check that $\log
  \in \mathcal{M} \left( I, \RR \right)$. This amounts showing that $h_{\ast}
  (x) := \log x - (x - 1)$ belongs to $\mathcal{B} \left( I, \RR \right)$.
  Since $h_{\ast} (1) = 0$, it remains to show that
  \begin{equation}
    \|h_{\ast} \|_{\mathcal{B}} = \sup_{x \in I \setminus \{1\}} 
    \frac{|h_{\ast} (x) | \cdot x}{(x - 1)^2} < + \infty .
  \end{equation}
  Using the Taylor expansion of $\log x$ at $x = 1$, we get $h_{\ast} (x) = -
  \frac{1}{2}  (x - 1)^2 + o ((x - 1)^2)$. Therefore
  \begin{equation}
    \frac{|h_{\ast} (x) | \cdot x}{(x - 1)^2}
  \end{equation}
  extends continuously to $x = 1$ with finite limit $1 / 2$. Being continuous
  on the compact interval $I$, the supremum is finite, confirming $h_{\ast}
  \in \mathcal{B} \left( I, \RR \right)$ and $\log \in \mathcal{M} \left( I,
  \RR \right)$.
  
  Because the Banach fixed-point theorem yields a unique fixed point of $T$ in
  $\mathcal{M} \left( I, \RR \right)$, and $\log$ is one such fixed point, it
  follows that $\log x$ is the unique fixed point of $T$ in $\mathcal{M} \left(
  I, \RR \right)$.
\end{remark}

\begin{proof}
  We divide the proof in two steps
  
  {\noindent}{{\em Step 1: The operator $T$ maps $\mathcal{M}$ into itself
  \/}}($T (\mathcal{M}) \subset \mathcal{M}$). Let $f \in \mathcal{M} \left(
  I, \RR \right)$. Then $f = f_0 + h$ for some $h \in \mathcal{B} \left( I,
  \RR \right)$. We compute
  \begin{equation}
    T [f] (x) - f_0 (x) = 2 (\sqrt{x} - 1 + h (\sqrt{x})) - (x - 1) = -
    (\sqrt{x} - 1)^2 + 2 h (\sqrt{x})
  \end{equation}
  Thus it is enough to show that both terms on the right-hand side belong to
  $\mathcal{B} \left( I, \RR \right)$. For the first term, define $g (x) = -
  (\sqrt{x} - 1)^2$. Then $g (1) = 0$, and for $x \neq 1$,
  \[ \|g\|_{\mathcal{B}} = \sup_{x \neq 1}  \frac{x}{(\sqrt{x} + 1)^2} . \]
  This function is continuous and clearly bounded on $I$ (approaching
  $1 / 4$ as $x \to 1$), hence $g \in \mathcal{B} \left( I, \RR \right)$. For
  the second term, let $h \in \mathcal{B} \left( I, \RR \right)$. Since
  $\sqrt{x} \in I$ whenever $x \in I$, the composite $x \mapsto h (\sqrt{x})$
  is well defined on $I$. In Step 2 below we show that the map $h \mapsto 2 h
  (\sqrt{\cdot})$ sends $\mathcal{B} \left( I, \RR \right)$ into itself. Hence
  $2 h (\sqrt{x}) \in \mathcal{B} \left( I, \RR \right)$. Therefore, $T [f] -
  f_0 \in \mathcal{B}$ and $T [f] \in \mathcal{M}$.{\smallskip}
  
  {\noindent}{{\em Step 2: Proof of Strict Contraction.\/}} Let $f_1, f_2 \in
  \mathcal{M} \left( I, \RR \right)$ and set $h = f_1 - f_2 \in \mathcal{B}
  \left( I, \RR \right)$. By linearity, $d (T [f_1], T [f_2]) = \|T
  [h]\|_{\mathcal{B}}$. We explicitly compute this norm:
  \[ \|T [h]\|_{\mathcal{B}} = \sup_{x \in I \setminus \{1\}}  \frac{2 |h
     (\sqrt{x}) | \cdot x}{(x - 1)^2} . \]
  Applying the substitution $y = \sqrt{x}$ and setting $J := [1 / \sqrt{A},
  \sqrt{A}] \subset I$, by the elementary inequality $(y + 1)^2 \geqslant 4
  y$, we obtain
  \[ \|T [h]\|_{\mathcal{B}} = \sup_{y \in J \setminus \{1\}} \left[ \frac{|h
     (y) | \cdot y}{(y - 1)^2} \cdot \frac{2 y}{(y + 1)^2} \right] \leqslant
     \frac{1}{2} \sup_{y \in I \setminus \{1\}} \frac{|h (y) | \cdot y}{(y -
     1)^2} = \frac{1}{2} \|h\|_{\mathcal{B}} . \]
  Therefore, $T$ is a strict global contraction on $\mathcal{B} \left( I, \RR
  \right)$ with contraction constant $\kappa = 1 / 2$. Finally, since
  $\mathcal{B} \left( I, \RR \right)$ is Banach, the affine space
  $\mathcal{M} \left( I, \RR \right) = f_0 +\mathcal{B} \left( I, \RR \right)$
  is complete with respect to the metric $d (f, g) = \|f - g\|_{\mathcal{B}}$.
  The Banach fixed-point theorem therefore applies and gives a unique fixed
  point of $T$ in $\mathcal{M}$.
  
  As observed in the remark, $\log x \in \mathcal{M}$ and satisfies $T [\log] =
  \log$. Hence this unique fixed point is precisely the natural logarithm.
\end{proof}

\section{Algorithms and Numerical Results}\label{sec:computational}
This section evaluates the numerical implementation of the methods developed in the preceding sections. We refer to the computational routines implementing the real exponential, complex exponential, and natural logarithm as \emph{kernels}.

\subsection{Overview}\label{sec:overview}
All three kernels follow a fixed-point construction pattern comprising three stages: argument reduction to a small interval, polynomial approximation (\emph{seed}) of the residual function on that interval, and reconstruction of the final value. The design space is governed by two parameters: the degree $K$ of the polynomial seed and the number $N$ of fixed-point iterations. Because each iteration halves the interval, a higher degree can be traded for more iterations at roughly equal accuracy. On general-purpose CPUs, high-degree/low-iteration configurations generally win, whereas architectures with limited multiplier width or storage favor short-seed/iterated configurations.

This fixed point approach replaces traditional table based argument reduction with purely arithmetic transformations, using repeated halvings for the exponential and repeated square roots for the logarithm. It relies exclusively on IEEE 754 arithmetic primitives, including fused multiply add (FMA) and direct exponent manipulation, without memory lookups, making it particularly well suited to vectorization.

The choice of seed differs structurally between functions. The residuals of the exponential and sine--cosine functions are entire; their Taylor series converge quickly, so simple truncated series are used. Conversely, the logarithm's residual converges slowly, requiring a fitted minimax polynomial to save degrees.

Accuracy is reported in ulp (units in the last place) against an 80-bit extended-precision reference. Speed is reported as a ratio to the SLEEF vector library (grades $u35$ and $u10$). All kernels are AVX2, table-free, and run in plain double precision unless marked double-double (dd). Note that these are research prototypes assuming finite, normal, in-range inputs, whereas SLEEF handles all IEEE special values; the comparison is therefore generous to our prototypes by construction, and its outcomes are read with this boundary in mind.

\subsection{Stability and Error Propagation}\label{sec:twoways}

The numerical stability of a fixed-point kernel depends heavily on the coordinate system used to represent its discrete orbit. While the underlying functional operators are contractive analytically, finite-precision arithmetic injects rounding errors that can either be amplified or attenuated during the reconstruction stage.

\subsubsection{The Exponential: Contractive Reconstruction}

The exponential kernel leverages the homomorphism $e^{a+b}=e^ae^b$ to reduce the argument to a small value $t_0$, evaluates a polynomial seed, and then ascends back to the original argument using $N$ successive doublings ($t \mapsto 2t$).
Crucially, the kernel computes this entirely within the residual coordinate $q(t)=(e^t-1-t)/t^2$ which removes the constant and linear contributions and isolates the nonlinear remainder. In this coordinate, the doubling map is represented by the rational recurrence
\begin{equation}\label{eq:qrecurrence}
D_q(q)=\frac14+\left(\frac12+\frac t2\right)q+\frac{t^2}{4}q^2.
\end{equation}
The reconstruction process is computationally contractive. Indeed, differentiating the recurrence with respect to $q$ and evaluating at the exact residual coordinate $q(t)$ gives
$ \partial_q D_q|_{q=q(t)}
={e^t}/{2}$. 
Since the argument reduction guarantees that the reduced argument remains in the interval $|t|\leqslant {\log 2}/{2}<1/2$, the contraction factor is uniformly bounded by ${e^{\log 2/2}}/{2}={\sqrt{2}}/{2}<1$. Therefore, perturbations introduced by floating point rounding during an ascent step are geometrically attenuated by subsequent doubling operations. The resulting error accumulation remains bounded, and the attainable accuracy is essentially independent of the number of reconstruction steps $N$.

In contrast, the naive approach of storing $y=e^t$ and repeatedly squaring it
($y\mapsto y^2$) has a fundamental drawback. Each squaring step doubles the
relative error, causing perturbations to grow by a factor of $2^N$ after $N$
iterations. The residual coordinate replaces this simple multiplication with a
slightly more expensive rational recurrence, but in return it ensures that the
reconstruction phase attenuates rather than amplifies the errors introduced at
each step (see Table~\ref{tab:amplification}).

\begin{remark}[Relation to scaling and squaring]
\label{rem:scaling}
The reconstruction phase is an instance of the classical \emph{scaling and squaring} paradigm underlying modern algorithms for the matrix exponential~\cite{higham2005scaling,almohy2009new}. The argument is first scaled by a power of two, an approximation is constructed on the reduced interval, and the original value is recovered by repeated squaring. Standard implementations of this paradigm, such as MATLAB's \texttt{expm}, execute this upward phase directly on the macroscopic matrix value ($X \mapsto X^2$). As for the scalar case (see Table~\ref{tab:amplification}), this macroscopic reconstruction amplifies injected rounding errors by a factor of two at each step. Consequently, current matrix algorithms must restrict the scaling depth and rely on high-degree Pad\'e approximants to control the backward error. Formulating the matrix ascent in a contractive residual coordinate (analogous to the scalar recurrence) could suppress this error amplification; we leave this as an open direction.
\end{remark}

\begin{table}[tbp]
\centering
\setlength{\arrayrulewidth}{0.8pt}
\renewcommand{\arraystretch}{1.25}
\begin{tabular*}{0.92\textwidth}{@{\extracolsep{\fill}} c c c c @{}}
\hline\noalign{\vskip 1pt}\hline\noalign{\vskip 2pt}
$N$ & \textsc{logarithm}, $T[f]=2f(\sqrt{\cdot})$
    & { \textsc{exponential}, naive $y\mapsto y^{2}$}
    & \textsc{exponential}, residual \\
\noalign{\vskip 2pt}\hline\noalign{\vskip 2pt}
0 & $1\times$   & $1\times$   & $1.00\times$ \\
1 & $2\times$   & $2\times$   & $0.59\times$ \\
2 & $4\times$   & $4\times$   & $0.32\times$ \\
3 & $8\times$   & $8\times$   & $0.17\times$ \\
4 & $16\times$  & $16\times$  & $0.09\times$ \\
5 & $32\times$  & $32\times$  & $0.04\times$ \\
6 & \textbf{$64\times$} & \textbf{$64\times$} & $0.02\times$ \\
\noalign{\vskip 2pt}\hline
\end{tabular*}
\vskip 4pt
\caption{\small{\emph{Error propagation through the reconstruction.} A seed
perturbation $\varepsilon_0=10^{-12}$ propagated through $N$ reconstruction
steps. The logarithm's ascent multiplies it by $2^N$; the naive exponential,
which stores the value $y=e^t$ and ascends by squaring, amplifies the relative
error by exactly the same factor $2^N$; the contractive residual coordinate
attenuates it instead. The middle column shows that the amplification is a
property of the coordinate, not of the function: the same exponential, in the
value coordinate, is precisely as unstable as the logarithm.}}
\label{tab:amplification}
\end{table}

\subsubsection{The Logarithm: Stabilization via the Product Form}
The logarithm is evaluated by first partitioning the input using its IEEE~754 representation. Any $x>0$ is expressed as $x = 2^{E}\,m$ with $m\in[1,2)$. A conditional scaling, $m\mapsto m/2$ and $E\mapsto E+1$ for $m\geqslant\sqrt{2}$, recenters the mantissa to the symmetric interval $m\in[2^{-1/2},2^{1/2})$. Folding one bit of the exponent into a half-integer count $e_h$ further restricts the domain to $m\in[2^{-1/4},2^{1/4})$. These operations modify only the exponent field and therefore incur no rounding error. The logarithm then decomposes additively as
\begin{equation}\label{eq:logsplit}
  \log x = e_h\,\tfrac{\log 2}{2} + \log m,
\end{equation}
reducing the problem to evaluating $\log m$ on a bounded interval near $1$, followed by the addition of the exponent term $e_h\,\tfrac{\log 2}{2}$.
Unlike the exponential kernel, where the recurrence attenuates rounding errors at each step, the logarithm iteration relies on the identity $\log m = 2^N \log (m^{1/2^N})$. Because the reconstruction of $\log m$ involves multiplying the polynomial approximation by the scalar $2^N$, numerical errors introduced during the approximation phase are amplified by this factor.

To mitigate this error growth, the iterative sequence is reformulated. The state variable is defined as the deviation from unity, $u_0 = m - 1$, rather than evaluating the argument $m$ directly. Because the recentered mantissa $m$ lies within $[0.5, 2.0]$, Sterbenz's lemma guarantees that this initial floating-point subtraction is computed exactly.

The fixed-point iteration computes a sequence of nested square roots, $s_k=m^{1/2^k}$, which converges to $1$. Tracking $s_k$ directly leads to numerical instability, as the final logarithm polynomial requires evaluating the difference $s_k-1$, which is subject to subtractive cancellation. To prevent this, the coordinate $u_k := s_k - 1$ is tracked at each step. Computing $u_N = s_N - 1$ directly at the end of the iteration would still incur subtractive cancellation, and the subsequent reconstruction would amplify the resulting error by $2^N$. This is resolved by applying the difference of squares, $s_k^2-1=(s_k-1)(s_k+1)$, yielding the recurrence $u_k = u_{k-1}/(s_k + 1)$. This leads to:
\begin{equation}
 u_N = \frac{u_0}{\prod_{k=1}^{N}(s_k+1)}
\end{equation} 
As $s_k \to 1$, each denominator factor approaches $2$. This sequence of divisions reduces the exact initial deviation $u_0$ by a factor of approximately $2^N$. The product formulation replaces subtractions with additions and divisions. Consequently, when the final reconstruction step scales this attenuated value by $2^N$, it offsets the prior division, preventing the error amplification that characterizes the direct evaluation (see Table~\ref{tab:forms}).

\begin{table}[tbp]
\centering
\setlength{\arrayrulewidth}{0.8pt}
\renewcommand{\arraystretch}{1.25}
\begin{tabular*}{0.8\textwidth}{@{\extracolsep{\fill}} c c c @{}}
\hline\noalign{\vskip 1pt}\hline\noalign{\vskip 2pt}
$N$ & \textsc{naive} & \textsc{product} \\
\noalign{\vskip 2pt}\hline\noalign{\vskip 2pt}
2 & 10.6  & \textbf{5.9} \\
3 & 18.4  & \textbf{5.9} \\
4 & 31.3  & \textbf{5.9} \\
5 & 58.2  & \textbf{5.9} \\
6 & 115.3 & \textbf{5.9} \\
\noalign{\vskip 2pt}\hline
\end{tabular*}
\vskip 4pt
\caption{\small{\emph{Naive versus product logarithmic descent.} Maximum absolute error ($\times10^{-16}$) of the scalar logarithm over $[10^{-3},10^{3}]$ with a degree-10 seed. The naive form grows exponentially, whereas the product form maintains a flat error floor.}}
\label{tab:forms}
\end{table}

\subsubsection{Functional Formulations and Pseudocode}
\label{sec:triptych}

Algorithms~\ref{alg:exp}, \ref{alg:cexp}, and \ref{alg:log} detail the pointwise execution paths. All three share the fixed-point architecture outlined in the overview (\S\ref{sec:overview}): argument reduction (Cody--Waite for the exponential family, quarter-octave for the logarithm), evaluation of a degree-$K$ polynomial seed, and $N$ reconstruction iterations, carried out in the residual coordinate for the exponential family and in the stabilized product form for the logarithm. The complex exponential additionally applies a final quadrant selection, mapping the result on the reduced argument $r\in[-\pi/4,\pi/4]$ back to the full circle.
The implementations rely exclusively on two standard, table-free arithmetic primitives:

\begin{itemize}

\item \textit{Cody--Waite Reduction:} Computes the reduced argument $r=x-kC$ (where $C$ is an irrational period such as $\log 2$ or $\pi/2$) without catastrophic cancellation. To avoid floating-point precision loss, we use standard techniques consisting of representing the constant in simulated extended precision as an unevaluated sum, $C \approx C_{\mathrm{high}} + C_{\mathrm{low}}$. This allows the dominant subtraction $x - k \cdot C_{\mathrm{high}}$ to be computed exactly, preserving the significant digits of the mathematical remainder before safely applying the small correction $k \cdot C_{\mathrm{low}}$.

    \smallskip

   \item \textit{Exponent Scaling ($\operatorname{scalbn}$):} Restores the
$2^k$ factor isolated during the reduction step. Instead of explicitly computing
$2^k$ and performing a floating-point multiplication, which introduces
additional rounding errors and may cause intermediate overflow, the integer
$k$ is directly added to the IEEE~754 binary exponent field of the operand.
This results in an exact scaling operation with $O(1)$ computational cost.
\end{itemize}

\noindent By utilizing these exact operations, the reduction and scaling phases avoid introducing additional rounding error, preserving the numerical accuracy of the core approximations detailed in the algorithms.

\begin{algorithm}[htbp]
\caption{Fixed-point real exponential (residual form)}
\label{alg:exp}
\begin{algorithmic}[1]
\Require $x \in \RR$; parameters $N, K \in \NN$
\Ensure $y \approx e^{x}$
\State $k \gets \operatorname{round}(x / \log 2)$
\State $r \gets \operatorname{fma}(k, -\log 2_{\mathrm{high}}, x) - k\log 2_{\mathrm{low}}$ \Comment{Cody--Waite two-limb reduction}
\State $t \gets r \cdot 2^{-N}$ \Comment{Exact bit-shift decrement of the floating-point exponent}
\State $q \gets \sum_{j=0}^K a_j t^j$ \Comment{Truncated Taylor series of $q(t)=(e^t - 1 - t)/t^2$, $a_j = 1/(j+2)!$}
\Loop\ ($N$ times)
  \State $a \gets \tfrac12 + \tfrac{t}{2}, \quad b \gets \tfrac{t^2}{4}$
  \State $q \gets \tfrac14 + a\,q + b\,q^2$ \Comment{Contractive residual recurrence $D_q$, Eq.~\eqref{eq:qrecurrence}}
  \State $t \gets 2t$ \Comment{Ascend one level toward $r$}
\EndLoop
\State \Return $\operatorname{scalbn} (\operatorname{fma}(r^2, q,\ 1+r),\ k)$ \Comment{Reconstruction $2^k\,(1+r+r^2 q)$}
\end{algorithmic}
\end{algorithm}

\begin{algorithm}[htbp]
\caption{Fixed-point complex exponential (sine and cosine together)}
\label{alg:cexp}
\begin{algorithmic}[1]
\Require $x \in \RR$; parameters $N, K \in \NN$
\Ensure $(\cos x, \sin x)$
\State $\kappa \gets \operatorname{round}(x \cdot 2/\pi) \bmod 4$ \Comment{Quadrant index}
\State $r \gets \textsc{CodyWaiteReduce}_{\pi/2}(x)$ \Comment{Three-limb reduction yielding $r \in [-\pi/4, \pi/4]$}
\State $t \gets r \cdot 2^{-N}$ \Comment{Exact scaling of the reduced argument}
\State $q_r \gets \sum_{j=0}^{\lfloor K/2 \rfloor} b_{2j} t^{2j}$ \Comment{Truncated Taylor seed: even part $q_r=(\cos t-1)/t^2$}
\State $q_i \gets \sum_{j=0}^{\lfloor (K-1)/2 \rfloor} b_{2j+1} t^{2j+1}$ \Comment{odd part $q_i=(\sin t-t)/t^2$}
\Loop\ ($N$ times) \Comment{Pointwise evaluation of the complex operator equation}
  \State $h_r \gets -0.25 + 0.5 q_r - 0.5 t q_i + 0.25 t^2 (q_r^2 - q_i^2)$
  \State $h_i \gets 0.5 q_i + 0.5 t q_r + 0.5 t^2 q_r q_i$
  \State $q_r \gets h_r, \quad q_i \gets h_i, \quad t \gets 2t$
\EndLoop
\State $c \gets 1 + r^2 q_r, \quad s \gets r + r^2 q_i$ \Comment{$(\cos r, \sin r)$ on the reduced argument}
\State \Return $\textsc{QuadrantSelect}(\kappa, c, s)$ \Comment{Recover $(\cos x, \sin x)$ from $(\cos r,\sin r)$ and the quadrant $\kappa$}
\end{algorithmic}
\end{algorithm}

\begin{algorithm}[htbp]
\caption{Fixed-point logarithm (product form)}
\label{alg:log}
\begin{algorithmic}[1]
\Require $x \in \RR^+$; parameters $N, K \in \NN$
\Ensure $y \approx \log x$
\State $(m, e_h) \gets \textsc{QuarterOctaveReduce}(x)$ \Comment{Decomposition mapping $m \in [2^{-1/4}, 2^{1/4})$}
\State $u_{0} \gets m - 1.0$ \Comment{Exact evaluation via Sterbenz's lemma}
\State $s \gets m, \quad D \gets 1.0$
\Loop\ ($N$ times)
  \State $s \gets \sqrt{s}, \quad D \gets D \cdot (s + 1.0)$ \Comment{Accumulate the telescoping denominator $D=\prod(s_k+1)$}
\EndLoop
\State $u \gets u_{0} / D$ \Comment{A single division executing range-reduction to the origin}
\State $P \gets \sum_{j=0}^K c_j u^j$ \Comment{Minimax polynomial seed of $\log(1+u)/u$}
\State \Return $\operatorname{fma}(e_h, \tfrac{\log 2}{2}, 2^N u P)$ \Comment{Reconstruction; $\tfrac{\log 2}{2}$ scales $e_h$, carried in half-integer units}
\end{algorithmic}
\end{algorithm}

\begin{table}[tbp]
\centering
\setlength{\arrayrulewidth}{0.8pt}
\renewcommand{\arraystretch}{1.05}
\begin{tabular*}{\textwidth}{@{\extracolsep{\fill}} l l c c l l @{}}
\hline\noalign{\vskip 1pt}\hline\noalign{\vskip 2pt}
\textsc{function} & \textsc{prec.} & $N$ & $K$ & \textsc{max ulp} & \textsc{vs sleef} \\
\noalign{\vskip 2pt}\hline\noalign{\vskip 2pt}
$\exp$      & plain & 0 & 12 & 0.67 \,($u10$) & \vf{$1.08\times$} $u10^{\dagger}$ \\
$\exp$      & plain & 1 & 10 & 0.65 \,($u10$) & $0.91\times$ $u10^{\dagger}$ \\
$\exp$      & plain & 2 &  8 & 0.66 \,($u10$) & $0.81\times$ $u10^{\dagger}$ \\
$\exp$      & plain & 3 &  7 & 0.66 \,($u10$) & $0.77\times$ $u10^{\dagger}$ \\
\noalign{\vskip 1pt}\hline\noalign{\vskip 2pt}
$\sin,\cos$ & plain & 0 &  8 & 1.42 \,($u10^{*}$) & \vf{$2.38\times$} $u35$,\; \vt{$3.85\times$}$u10$\\
$\sin,\cos$ & plain & 1 &  6 & 1.42 \,($u10^{*}$) & \vf{$1.88\times$} $u35$,\; \vt{$3.05\times$}$u10$\\
$\sin,\cos$ & plain & 2 &  5 & 1.44 \,($u10^{*}$) & \vf{$1.49\times$} $u35$,\; \vt{$2.41\times$}$u10$\\
$\sin,\cos$ & plain & 3 &  4 & 2.21 \,($u35$)         & \vf{$1.24\times$} $u35$,\; \vt{$2.00\times$}$u10$\\
\noalign{\vskip 1pt}\hline\noalign{\vskip 2pt}
$\log$      & plain & 0 & 15 & 1.99 \,($u35$) & \vf{$1.64\times$} $u35$,\; \vt{$2.21\times$}$u10$\\
$\log$      & plain & 1 & 12 & 3.21 \,($u35$) & \vf{$1.27\times$} $u35$,\; \vt{$1.72\times$}$u10$\\
$\log$      & plain & 2 & 10 & 3.54 \,($u35$) & \vt{$1.16\times$} $u35$,\; \vt{$1.57\times$}$u10$\\
$\log$      & plain & 3 & 10 & 3.32 \,($u35$) & $0.85\times$ $u35$,\; \vt{$1.14\times$}$u10$\\
$\log$      & dd    & 0 & 14 & 0.59 \,($u10$) & \vf{$1.16\times$} $u35$,\; \vf{$1.57\times$}$u10$\\
\noalign{\vskip 2pt}\hline
\end{tabular*}

\vskip 4pt
\begin{minipage}{1\textwidth}

\caption{\small{\emph{Fixed-point kernels versus SLEEF (AVX2).} Accuracy is reported as maximum ulp error, categorized by grade: $u10$($\leqslant1$), $u35$ ($\leqslant3.5$), and $u10^{*}$ (just above $u10$). Speed ratios are measured against the respective SLEEF grade ($^{\dagger}$the $\exp$ kernel compares to $u10$, as SLEEF lacks a $u35$ routine). \vf{Full bordeaux} indicates configurations that are faster and meet the target grade; \vt{light bordeaux} indicates faster configurations that exceed the error threshold; plain black denotes slower performance. Seeds are Taylor polynomials for $\exp$ and $\sin,\cos$, and minimax for $\log$. As established in \S\ref{sec:twoways}, the rising error floor of the $\log$ kernel reflects its amplifying macroscopic reconstruction.}}
\label{tab:vector}
\end{minipage}
\end{table}

\subsection{The Vector Regime: Throughput Characteristics}\label{sec:vector}

Under SIMD vectorization, where a single instruction operates concurrently across multiple register lanes (e.g., four double-precision operands), eliminating table lookups avoids irregular memory traffic. Rather than paying the latency of gather operations at data-dependent addresses, table-free kernels run entirely within registers.

To evaluate this approach, the AVX2-vectorized prototypes were benchmarked against the SLEEF library. As detailed in \S\ref{sec:overview}, accuracy is reported in units in the last place (ulp) against an 80-bit reference, and speed is measured as a relative throughput ratio against SLEEF's $u10$ ($\leqslant 1$ ulp) and $u35$ ($\leqslant 3.5$ ulp) baselines.

The polynomial seeds approximate the analytic residuals of each coordinate system: $q(t)=(e^{t}-1-t)/t^2$ for the real exponential, $q_r(y)=(\cos y-1)/y^2$ and $q_i(y)=(\sin y-y)/y^2$ for the complex exponential, and a minimax approximation of $\log(1+u)/u$ for the logarithm. The logarithmic residual is evaluated on the reduced mantissa $m=1+u$, where the deviation $u=m-1$ is computed exactly. To preserve relative accuracy near $m=1$ (where $u \to 0$), we employ the standard technique of factoring out the root. Instead of approximating the logarithm directly, we approximate the well-conditioned quotient $\log(1+u)/u$ with a polynomial. Multiplying this polynomial by $u$ recovers the logarithm and guarantees bounded relative error, whereas a direct polynomial approximation of $\log(1+u)$ would leave a tiny non-zero absolute error at $u=0$, causing the relative error to diverge.

Table~\ref{tab:vector} indicates that the performance and accuracy of these portable, table-free prototypes are comparable to those of a standard production library. Their performance and accuracy profiles depend on the algebraic structure of each function's reconstruction stage:

\begin{itemize}
    \item \textit{Real Exponential (No Cancellation):} Reconstructed via $e^{r}=1+r+r^2q$, the calculation avoids subtractive cancellation. This preserves relative precision, allowing standard double precision to achieve $u10$ accuracy ($\approx0.65$ ulp). Because the reconstruction is contractive, the error floor remains constant across iteration depths $N \in \{0, 1, 2, 3\}$. This decoupling permits the selection of a design point according to hardware constraints: $N=0$ maximizes throughput on general-purpose CPUs ($1.08\times$ SLEEF), while higher $N$ values reduce coefficient storage requirements on constrained architectures.
    \smallskip

    \item \textit{Sine and Cosine (Local Cancellation):} Reconstructed via $\cos x=1+r^2q_r$ and $\sin x=r+r^2q_i$, local cancellation near the zeros limits the global precision to approximately $1.42$ ulp (denoted as $u10^{*}$). Away from these zeros, the same seed yields an error of $0.68$ ulp, indicating that the error bound is determined by the algebraic structure rather than the polynomial approximation. This kernel yields higher throughput than SLEEF across all evaluated iteration depths, showing a $2\times$ to $4\times$ increase compared to the strict $u10$ baseline, with an accuracy reduction of approximately $0.5$ ulp.
    \smallskip

    \item \textit{Logarithm (Macroscopic Shift and Amplification):} Reconstructed via $\log x=e_h\,\tfrac{\log 2}{2}+\log m$, the overall precision is limited by the addition of the exponent term $e_h\,\tfrac{\log 2}{2}$. In standard double precision, this operation limits accuracy to approximately $2$ ulp at $N=0$. As the iteration depth $N$ increases, this reconstruction structure amplifies intermediate rounding errors by a factor of $2^N$. Consequently, the error increases ($1.99 \to 3.54$ ulp) and exceeds the $u35$ threshold at $N=2$. Achieving $u10$ accuracy ($0.59$ ulp) requires evaluating the addition in double-double (dd) precision at $N=0$; this configuration maintains a higher throughput than the corresponding SLEEF baselines.
    
    \item \textit{Logarithm (Macroscopic Shift and Amplification):} Reconstructed via $\log x=E\log 2+\log m$, the overall precision is limited by the addition of the exponent term $E\log 2$. In standard double precision, this operation limits accuracy to approximately $2$ ulp at $N=0$. As the iteration depth $N$ increases, this reconstruction structure amplifies intermediate rounding errors by a factor of $2^N$. Consequently, the error increases ($1.99 \to 3.54$ ulp) and exceeds the $u35$ threshold at $N=2$. Achieving $u10$ accuracy ($0.59$ ulp) requires evaluating the addition in double-double (dd) precision at $N=0$; this configuration maintains a higher throughput than the corresponding SLEEF baselines.
\end{itemize}

\noindent In summary, the vector implementation characteristics follow a consistent pattern within the $(K, N)$ design space. The iteration depth $N$ and the polynomial degree $K$ can be adjusted inversely while maintaining approximately constant accuracy. This allows the kernels to be adapted to specific hardware constraints, such as balancing multiplier utilization against register usage. However, regardless of the chosen $(K, N)$ configuration, the algebraic structure of the final reconstruction formula determines the minimum achievable error in finite precision.

\begin{table}[!t]
\centering
\setlength{\arrayrulewidth}{0.8pt}
\renewcommand{\arraystretch}{1.05}
\begin{tabular*}{0.86\textwidth}{@{\extracolsep{\fill}} l l c c c l @{}}
\hline\noalign{\vskip 1pt}\hline\noalign{\vskip 2pt}
\textsc{function} & \textsc{prec.} & $N$ & $K$ & \textsc{max error} & \textsc{vs sleef} \\
\noalign{\vskip 2pt}\hline\noalign{\vskip 2pt}
$\exp$      & plain & 1 & 2 & $3.0\times10^{-6}$ (rel.) & \vt{$2.01\times$} $u10^{\dagger}$ \\
$\exp$      & plain & 2 & 2 & $1.7\times10^{-7}$ (rel.) & \vf{$1.49\times$} $u10^{\dagger}$ \\
$\exp$      & plain & 3 & 2 & $1.1\times10^{-8}$ (rel.) & \vf{$1.18\times$} $u10^{\dagger}$ \\
$\exp$      & plain & 4 & 2 & $6.5\times10^{-10}$ (rel.) & $0.98\times$ $u10^{\dagger}$ \\
$\exp$      & plain & 1 & 4 & $2.2\times10^{-9}$ (rel.) & \vf{$1.80\times$} $u10^{\dagger}$ \\
$\exp$      & plain & 2 & 4 & $3.1\times10^{-11}$ (rel.) & \vf{$1.39\times$} $u10^{\dagger}$ \\
$\exp$      & plain & 3 & 4 & $4.7\times10^{-13}$ (rel.) & \vf{$1.10\times$} $u10^{\dagger}$ \\
$\exp$      & plain & 4 & 4 & $7.3\times10^{-15}$ (rel.) & $0.90\times$ $u10^{\dagger}$ \\
\noalign{\vskip 1pt}\hline\noalign{\vskip 2pt}
$\sin,\cos$ & plain & 1 & 2 & $9.2\times10^{-6}$ (abs.) & \vt{$2.42\times$} $u35$ \\
$\sin,\cos$ & plain & 2 & 2 & $2.6\times10^{-7}$ (abs.) & \vf{$1.83\times$} $u35$ \\
$\sin,\cos$ & plain & 3 & 2 & $7.6\times10^{-9}$ (abs.) & \vf{$1.39\times$} $u35$ \\
$\sin,\cos$ & plain & 4 & 2 & $2.3\times10^{-10}$ (abs.) & \vf{$1.11\times$} $u35$ \\
$\sin,\cos$ & plain & 1 & 4 & $4.4\times10^{-11}$ (abs.) & \vf{$2.15\times$} $u35$ \\
$\sin,\cos$ & plain & 2 & 4 & $7.7\times10^{-14}$ (abs.) & \vf{$1.61\times$} $u35$ \\
$\sin,\cos$ & plain & 3 & 4 & $2.3\times10^{-16}$ (abs.) & \vf{$1.24\times$} $u35$ \\
$\sin,\cos$ & plain & 4 & 4 & $1.3\times10^{-16}$ (abs.) & \vf{$1.01\times$} $u35$ \\
\noalign{\vskip 1pt}\hline\noalign{\vskip 2pt}
$\log$      & plain & 1 & 4 & $2.5\times10^{-7}$ (abs.) & \vf{$1.96\times$} $u35$ \\
$\log$      & plain & 2 & 4 & $1.6\times10^{-8}$ (abs.) & \vf{$1.25\times$} $u35$ \\
$\log$      & plain & 3 & 4 & $1.9\times10^{-9}$ (abs.) & $0.92\times$ $u35$ \\
$\log$      & plain & 4 & 4 & $1.0\times10^{-9}$ (abs.) & $0.72\times$ $u35$ \\
$\log$      & plain & 1 & 5 & $9.6\times10^{-9}$ (abs.) & \vf{$1.86\times$} $u35$ \\
$\log$      & plain & 2 & 5 & $1.2\times10^{-9}$ (abs.) & \vf{$1.25\times$} $u35$ \\
$\log$      & plain & 3 & 5 & $9.6\times10^{-10}$ (abs.) & $0.91\times$ $u35$ \\
$\log$      & plain & 4 & 5 & $9.5\times10^{-10}$ (abs.) & $0.72\times$ $u35$ \\
\noalign{\vskip 2pt}\hline
\end{tabular*}

\vskip 4pt
\begin{minipage}{0.86\textwidth}
\caption{\small{\emph{Short-seed configurations for the $10^{-6}$-accuracy regime.} Metrics denote relative error for $\exp$ (whose output spans many binades) and absolute error for $\sin,\cos$, and $\log$ (whose outputs are of moderate size and cross zero). Two seed degrees ($K$) are evaluated across iteration depths $N=1,\dots,4$. For the contractive kernels, the error falls geometrically by $\approx2^{-(K+2)}$ per step, whereas the $\log$ kernel hits an amplification floor near $10^{-9}$ (\S\ref{sec:twoways}). Ratios compare against SLEEF's lowest-accuracy tier ($^{\dagger}$u10 for $\exp$). \vf{Full bordeaux} marks configurations that are faster and meet the $10^{-6}$ target; \vt{light bordeaux} marks faster configurations whose error exceeds the target.}}\label{tab:fast}
\end{minipage}
\end{table}

\subsubsection{A Reduced-Accuracy, High-Throughput Regime}
While standard production libraries define accuracy targets for general numerical
computing, specific applications such as computer graphics, signal processing,
and machine learning can operate with error tolerances near $10^{-6}$
(comparable to single precision). Standard double-precision libraries typically
do not provide performance optimizations for this reduced-accuracy regime.

The fixed-point methods can target this regime by reducing the degree of the polynomial seed, $K$. As shown in Table~\ref{tab:fast}, increasing the iteration depth $N$ with these lower-degree seeds reduces the approximation error by a factor of approximately $2^{-(K+2)}$ per step for the contractive kernels (exponential and sine/cosine). For the logarithm, the error bounds near $10^{-9}$ when the polynomial approximation error intersects the error amplification inherent to its reconstruction phase. Adjusting the $(K,N)$ parameters allows the kernels to traverse the accuracy-throughput design space.

\subsubsection{The $(K,N)$ Trade-Off as a Hardware Portability Axis}\label{sec:tradeoff}

Inspection of Table~\ref{tab:vector} reveals an accuracy preserving frontier in
the $(K,N)$ parameter space. For the real exponential kernel, configurations
ranging from $(N{=}0,K{=}12)$ to $(N{=}3,K{=}7)$ achieve approximately
$0.65$~ulp accuracy. This suggests that the degree of the polynomial seed
approximation can be reduced by increasing the number of reconstruction
iterations, without affecting the final accuracy. This observation can also be
justified by the formal error estimates developed above, although we do not
pursue this analysis further here.

The existence of this frontier is a direct consequence of the contractive
residual coordinate. Since the reconstruction phase attenuates rather than
amplifies rounding errors, increasing the number of iterations affects the
computational cost but does not significantly alter the attainable accuracy.
Consequently, configurations along the frontier are numerically equivalent but
have different computational characteristics, allowing the implementation to
select the most suitable configuration for the target architecture.

\begin{itemize}
   \item \textit{General purpose CPUs (high $K$, low $N$):} On architectures where
    polynomial coefficients can be stored in registers and evaluated with low
    latency, it is preferable to move along the frontier toward higher $K$ and
    fewer iterations, since the extra coefficients are essentially free while each
    iteration adds latency (see the throughput ratios in
    Tables~\ref{tab:vector} and~\ref{tab:fast}).

    \smallskip

    \item \textit{FPGAs and embedded platforms (low $K$, high $N$):} On
    architectures with limited coefficient storage, register availability, or
    multiplier resources, it is preferable to move toward lower $K$, compensating
    with a higher iteration count $N$ to stay on the frontier. This trades
    coefficient storage for repeated arithmetic, reducing the memory footprint at
    the cost of additional iterations (see the throughput ratios in
    Tables~\ref{tab:vector} and~\ref{tab:fast}).
\end{itemize}

\noindent This trade-off is conceptually similar to CORDIC, exchanging a complex initial approximation for additional iterative steps to maintain accuracy. The underlying mechanisms differ fundamentally. CORDIC relies on precomputed look-up tables, while the present approach reconstructs the function value via a contractive rational recurrence in the residual coordinate, requiring no tables. This contractive property attenuates rounding errors from one step to the next, making the attainable accuracy essentially independent of the iteration count $N$. In contrast, a direct reconstruction would amplify rounding errors, causing accuracy to degrade with $N$ (see Table~\ref{tab:amplification}).

\begin{table}[tbp]
\centering
\setlength{\arrayrulewidth}{0.8pt}
\renewcommand{\arraystretch}{1.05}
\begin{tabular*}{0.9\textwidth}{@{\extracolsep{\fill}} l c c c c l @{}}
\hline\noalign{\vskip 1pt}\hline\noalign{\vskip 2pt}
\textsc{coordinate} & $N$ & $K$ & \textsc{max ulp} & \textsc{vs $u35$}$^{\ddagger}$ & \textsc{vs $u10$}$^{\ddagger}$ \\
\noalign{\vskip 2pt}\hline\noalign{\vskip 2pt}
residual $\rho$ & 0 & 8  & $4.9\times10^{6}$ & $1.55\times$ & $2.41\times$ \\
residual $\rho$ & 1 & 8  & $1.5\times10^{3}$ & $0.90\times$ & $1.40\times$ \\
residual $\rho$ & 2 & 8  & $1.50$            & $0.50\times$ & $0.78\times$ \\
residual $\rho$ & 3 & 8  & $1.50$            & $0.34\times$ & $0.52\times$ \\
residual $\rho$ & 4 & 8  & $1.50$            & $0.25\times$ & $0.39\times$ \\
\noalign{\vskip 1pt}\hline\noalign{\vskip 2pt}
residual $\rho$ & 0 & 12 & $4.6\times10^{3}$ & $1.27\times$ & $1.97\times$ \\
residual $\rho$ & 1 & 12 & $1.50$            & $0.80\times$ & $1.24\times$ \\
residual $\rho$ & 2 & 12 & $1.50$            & $0.49\times$ & $0.75\times$ \\
residual $\rho$ & 3 & 12 & $1.50$            & $0.33\times$ & $0.51\times$ \\
residual $\rho$ & 4 & 12 & $1.50$            & $0.25\times$ & $0.39\times$ \\
\noalign{\vskip 2pt}\hline
\end{tabular*}
\vskip 4pt
\caption{\small{\emph{Preliminary simulation of the logarithm in the contractive residual coordinate $\rho$.} Unlike the value-coordinate formulation, the error stabilizes at a flat floor ($\approx 1.5$ ulp) as the iteration depth $N$ increases, demonstrating damped error propagation. The rows far from the floor ($N=0$; $N=1$ at $K=8$) reflect that the residual approach relies on the iteration: without descent, a short seed cannot cover the reduced interval. $^{\ddagger}$Throughput ratios are unoptimized (evaluating standard square roots in plain form) and are reported only to indicate order of magnitude relative to SLEEF.}}
\label{tab:logrho}
\end{table}

\subsection{Discussion}\label{sec:synthesis}

The preceding sections analyzed the numerical accuracy and SIMD performance of the proposed fixed-point algorithms. We conclude by summarizing the principal algorithmic trade-offs, discussing the current limitations, and outlining directions for future work.

\subsubsection{Performance Trade-offs}\label{sec:role}

As discussed in \S\ref{sec:tradeoff}, the parameters $(K,N)$ define a family of implementations that exchange polynomial complexity for iteration depth while maintaining essentially the same accuracy. For the real exponential, configurations ranging from $(K,N)=(12,0)$ to $(7,3)$ all achieve a maximum error of approximately $0.65$--$0.67$ ulp. The choice of $(K,N)$ may therefore be adapted to the characteristics of the target architecture without sacrificing numerical accuracy.

On the CPUs considered here, the throughput optimum lies at the low-$N$ end of the frontier ($N=0$, reaching $1.08\times$ the SLEEF $u10$ throughput). Configurations with smaller polynomial seeds and deeper iteration become attractive on architectures where polynomial evaluation is comparatively expensive, such as extended-precision arithmetic or memory-constrained hardware.

\subsubsection{Future Work: A Residual Formulation for the Logarithm}\label{sec:invitation}

Natural engineering extensions include implementations for additional SIMD instruction sets, such as AVX-512 and ARM Neon, together with branchless handling of IEEE-754 exceptional values.

A more significant algorithmic direction concerns the logarithm. Throughout this work, the logarithm is reconstructed from its macroscopic value. As discussed in \S\ref{sec:twoways}, this reconstruction amplifies accumulated rounding errors by a factor of $2^N$. This amplification is not an intrinsic property of the logarithm, but rather a consequence of the chosen coordinate. As for the exponential, one may instead introduce the residual coordinate
\begin{equation}
\rho(u)=\frac{\log(1+u)-u+\frac12u^2}{u^3},
\end{equation}
so that $\log(1+u)=u-\frac12u^2+u^3\rho(u)$. Setting $v:=\sqrt{1+u}-1$, one gets $\log(1+u)=2\log(1+v)$,
which yields
\begin{equation}
\rho(u)=\frac{2\rho(v)+2+\frac12v}{(2+v)^3}.
\end{equation}
\noindent The Taylor expansion of $\log(1+u)$ around zero gives
$\rho(u)=\frac13+O(u)$, showing that the residual remains bounded as $u\to0$ and admits the continuous extension $\rho(0)=1/3$. The important stability property, however, is the propagation of perturbations through the fixed-point iteration. Since the contribution of an error in $\rho(v)$ to the updated value of $\rho(u)$ is multiplied by
$\frac{2}{(2+v)^3}=\frac14+O(v)$,
small perturbations are reduced by approximately a factor of four when $u$ is close to zero. Thus, the residual formulation damps rounding errors during the iteration rather than amplifying them, in contrast with the macroscopic formulation, where reconstruction errors grow proportionally to $2^N$.

A rigorous convergence analysis of this operator, for example in an appropriate weighted Banach space, remains the subject of future work. Nevertheless, preliminary numerical experiments exhibit the expected behavior. As shown in Table~\ref{tab:logrho}, increasing the iteration depth rapidly reduces the approximation error, after which the maximum error stabilizes at approximately $1.5$ ulp, with no evidence of the exponential error growth observed in the value-coordinate formulation.

These preliminary results indicate that the residual formulation restores the algorithmic symmetry between the exponential and logarithm. Achieving full $u10$ accuracy will likely require carrying the final reconstruction, $E\log 2+\log m$, in extended precision.

\subsubsection{Concluding Remarks}\label{sec:closing}

The numerical stability and computational efficiency of the proposed algorithms follow directly from the interaction between their fixed-point structure and finite-precision arithmetic. SIMD architectures are particularly well suited to these methods because they replace table lookups with arithmetic operations that can be efficiently vectorized.

The choice of coordinate determines the numerical behavior. For the exponential, performing the iteration in a contractive residual coordinate suppresses the propagation of rounding errors. For the logarithm, reconstructing the function from its macroscopic value requires explicit stabilization, whereas performing the iteration directly in the residual coordinate recovers the same contractive mechanism as for the exponential, providing a unified framework for both elementary functions.

The main contribution of this work is the identification of the residual coordinate as the fundamental mechanism that eliminates the $2^N$ amplification of rounding errors. This is a structural property of the underlying algorithm rather than of a particular implementation, and we expect it to remain relevant across different hardware platforms and software environments. Fully exploiting its practical potential will require the extensive architecture-specific optimization, engineering effort, and benchmarking that have gone into the development of state-of-the-art mathematical libraries. We hope that the present work provides a sound theoretical foundation and a practical point of departure for researchers and developers working on high-performance mathematical software, both in academia and in industry, and that it encourages further interaction between the mathematical analysis, numerical analysis, and high-performance computing communities in the development of faster, more reliable, and more energy-efficient elementary function libraries.

\section*{Declarations}

\noindent
\textit{Author Contributions:} All authors contributed equally to this work.
\smallskip

\noindent
\textit{Ethical Approval:} Not applicable. This study does not involve human participants or animals.
\smallskip

\noindent
\textit{Conflict of Interest:} The authors declare that they have no conflicts of interest.
\smallskip

\noindent
\textit{Data and Code Availability:} No datasets were generated or analyzed during the current study. The C++ source code and benchmarking kernels evaluated in this work are currently available from the corresponding author upon reasonable request. The authors plan to clean, document, and host the complete codebase in a publicly accessible GitHub repository before final publication.
\smallskip

\noindent
\textit{Declaration of Generative AI and AI-Assisted Technologies:}
During the preparation of this manuscript, the authors used Grammarly and Gemini to check spelling and improve the fluency of selected portions of the text. The manuscript was initially prepared using TeXmacs and subsequently exported to LaTeX. Gemini was used to identify and correct translation or conversion errors arising during this process, to assist with LaTeX formatting, and to generate LaTeX code to improve the formatting of tables that were not converted correctly during the export.

In addition, Claude was used to assist in implementing and testing the algorithms described in the paper.

No generative AI tools were used in the development of the mathematical ideas, the derivation of the theoretical results, or the proofs. The authors are fully responsible for the content of the manuscript and for the accuracy, validity, and integrity of all results presented.

\section*{Acknowledgements}
\textsc{G.D.F.} is a member of GNAMPA--INdAM. He gratefully acknowledges partial financial support from the GNAMPA Project CUP\_E53C25002010001, and from the University of Naples Federico II through the FRA Project-B ``VarMoCry'' on  \emph{Variational Analysis and Modeling of Liquid Crystals}. Further support is acknowledged from the Italian Ministry of University and Research through the PRIN 2022 project \emph{Variational Analysis of Complex Systems in Material Science, Physics and Biology} (No.~2022HKBF5C).

\bibliographystyle{alpha}
\bibliography{Eix_Bib}

\end{document}